\documentclass{article}

\usepackage{amssymb,amsfonts, stmaryrd, amsmath, amsthm}
\usepackage[all,arc]{xy}
\usepackage{enumerate}
\usepackage{graphicx, enumerate}
\usepackage[tmargin=0.75in,bmargin=0.75in,lmargin=0.75in,rmargin=0.75in]{geometry}
\usepackage[utf8]{inputenc}
\usepackage[parfill]{parskip}

\makeatletter
\def\thm@space@setup{%
  \thm@preskip=\parskip \thm@postskip=0pt
}
\makeatother

\newtheorem{thm}{Theorem}[section]
\newtheorem*{thm*}{Theorem}
\newtheorem{cor}[thm]{Corollary}
\newtheorem*{cor*}{Corollary}
\newtheorem{prop}[thm]{Proposition}
\newtheorem{lem}[thm]{Lemma}

\theoremstyle{definition}
\newtheorem{defn}[thm]{Definition}

\theoremstyle{remark}

\newtheorem*{ack*}{Acknowledgements}

\newcommand{\Z}{\mathbb{Z}}

\newenvironment{bsmallmatrix}
  {\left[\begin{smallmatrix}}
  {\end{smallmatrix}\right]}

\title{Conic intersections, Maximal Cohen-Macaulay modules and the Four Subspace problem}

\author{Vincent G\'elinas \thanks{The author is affiliated with the University of Toronto, and is supported by an NSERC CGS-D Grant.}}
\renewcommand\footnotemark{}

\begin{document}

\maketitle

\begin{abstract}
\noindent Let $X$ be a set of $4$ generic points in $\mathbb{P}^2$ with homogeneous coordinate ring $R$. We classify indecomposable graded MCM modules over $R$ by reducing the classification to the Four Subspace problem solved by Nazarova and Gel$'$fand-Ponomarev, or equivalently to the representation theory of the $\widetilde{D}_4$ quiver. In particular, the $\mathbb{P}^1$ tubular family of regular representations corresponds to matrix factorizations of the pencil of conics going through $X$, with smooth conics $Q_{t}$ corresponding to rank one tubes and the singular conics $Q_0, Q_1, Q_{\infty}$ giving the remaining rank two tubes. As applications we determine the Ulrich modules over $R$ and we identify the preprojective algebra of type $\widetilde{D}_4$ as the diagonal part of the Yoneda algebra of a Koszul $R$-module.
\end{abstract}

\tableofcontents

%


\section{Introduction}

There is a long tradition in algebraic geometry and commutative algebra of reducing classification problems to `matrix problems', to then be attacked by the methods of the representation theory of algebras. A striking example of this is the classification of indecomposable stable graded modules over the exterior algebra $\Lambda(dx, dy)$ on two generators  by Kronecker's theory of matrix pencils. The aim of this paper is to present a new such classification, by reducing an elementary problem of commutative algebra to the Nazarova-Gelfand-Ponomarev classification of four subspaces configurations.

More specifically, we solve the following problem. Let $X$ be a set of four distinct points in $\mathbb{P}_k^2$ with no three collinear and $k$ an algebraically closed field, and let $R = k[X]$ be the homogeneous coordinate ring. Let $Q$ be a $\widetilde{D}_4$ quiver and $A = kQ$ its path algebra. We reduce the classification of stable graded Maximal Cohen-Macaulay (MCM) modules over $R$ to the classification of quiver representations over $Q$.

\begin{thm}\label{E}
There is an equivalence of triangulated categories ${\rm E}: \underline{\rm MCM}(gr R)^{op} \xrightarrow{\cong} {\rm D}^{b}(kQ)$.
\end{thm}

Theorems of this type have been heavily studied in the last decade \cite{MR3063907}, \cite{MR2313537}, \cite{MR3498794}, \cite{MR3107527}, where such equivalences follow from the standard methods of tilting theory. One drawback to the method is that the inverse to ${\rm E}$ is formally given and thus difficult to calculate with; for instance ${\rm E}^{-1}$ typically produces large non-minimal presentations of indecomposables which may have a more economical description.


The main contribution of this paper is to give, in characteristic different from $2$, a complete, explicit list of stable indecomposable graded MCM modules over $R$, which correspond under ${\rm E}$ to the indecomposable complexes of representations of $Q$. The classification is of a rather compelling geometric nature, relying on the study of the pencil of quadrics through the points $X$ in $\mathbb{P}^2$, and can be stated without reference to quiver representations. 

Recall that a module has complexity $c$ if its minimal resolution has polynomial growth asymptotic to $Cn^{c-1}$ for some $C > 0$. We write $N^{st}$ for the MCM approximation of an $R$-module $N$.
\begin{thm}\label{maintheorem} Assume that char $k \neq 2$. Then up to degree shift, the indecomposable graded MCM modules $M$ are classified as follows:

\begin{itemize}
\item[] \underline{Complexity $2$}: $M$ is a (co)syzygy of either $k^{st}$ or $L_i$, where $L_i = \Gamma_{\geq 0}(X, \mathcal{O}_{p_i}) = k[X]/I(p_i)$ is the point module corresponding to $p_i \in X \subset \mathbb{P}^2$.
\item[] \underline{Complexity $1$}: $M$ is given by a matrix factorization of some quadratic polynomial $Q_t$ coming from the unique pencil of conics $\{ V(Q_t) \}_{t \in \mathbb{P}^1}$ passing through $X$ in $\mathbb{P}^2$, as listed below. The modules $N_{t}\langle r \rangle$ correspond to matrix factorizations of non-degenerate conics while $D_{t}\langle r \rangle^{+}$, $D_{t} \langle r \rangle^{-}$ come from degenerate conics.\\

\centerline{
\begin{tabular}{|c|c|c|}
  \hline
   \multicolumn{3}{|c|}{Indecomposable graded MCM modules (up to degree shift)}\\
   \hline
 Complexity $2$ & \multicolumn{2}{c|}{Complexity $1$}\\
 \hline
 & $t \in \mathbb{P}^1 \setminus \{0, 1, \infty \}$ & $t = 0, 1, \infty$ \\
 \cline{2-3}
 & & \\
 $syz_R^{n}\big(k^{st} \big)$, $n \in \Z$ & $N_{t}\langle r \rangle$, $r \geq 1$ & $D_{t}\langle r \rangle^{+}$, $r \geq 1$ \\
$syz_R^{n}\big(L_i \big)$, $n \in \Z$ &  & $D_{t}\langle r \rangle^{-}$, $r \geq 1$ \\
& & \\
 \hline
\end{tabular}}

The modules $N_{t}\langle r \rangle$ have $1$-periodic minimal resolutions while the minimal resolutions of $D_{t}\langle r \rangle^{\pm}$ have period $2$.
\end{itemize}

\end{thm}

We will see that all MCM modules over the completion $\widehat{R}$ are gradable, and so these are classified as well by the above list. We describe all indecomposables in section $3$. In particular we list the Betti tables and numerical invariants of each indecomposable MCM module, and classify the Ulrich modules.

Lastly, let $\Pi(Q)$ be the preprojective algebra of type $Q = \widetilde{D}_4$. From Theorem \ref{E} we extract:
\begin{thm*}
There is an isomorphism of graded algebras $\Pi(Q) \cong {\rm Ext}^{{\rm diag}}_R(U, U)$ for a specific Koszul module $U$.
\end{thm*}

The structure of the paper is as follows. We review standard background in section $2$ and present the classification of representations of $Q$ in the form we will need. Note that as the setting of this paper involves at most elementary commutative algebra, some effort was made to keep the work self-contained and accessible. We then lay out the analogous classification of MCM modules in section $3$ but wait until section $4$ to prove Theorems \ref{E} and \ref{maintheorem}. The relationship with the preprojective algebra is investigated in the last section.\\

\begin{ack*}
I wish to thank Benjamin Briggs and \"Ozg\"ur Esentepe for careful readings of this paper and suggestions toward improving the presentation. Thanks are also due to Ragnar-Olaf Buchweitz for feedback and for encouraging me to write up this note.
\end{ack*}

\section{Background}


Throughout this paper, let $k$ be an algebraically closed field, which will eventually be of characteristic not $2$ starting in section $3$. All modules will be finitely generated, and are taken to be left modules when the ring is noncommutative. For graded modules $M, N$ over a graded ring $R$, we write ${\rm Hom}_{gr R}(M, N) = {\rm Hom}_R^0(M, N)$ for the homogeneous morphisms of degree zero. Let $M(i)$ stand for the internal shift $M(i)_n = M_{n+i}$. When $R$ is Noetherian, the isomorphism $\bigoplus_{i \in \mathbb{Z}} {\rm Ext}^*_{gr R}(M, N(i)) \cong {\rm Ext}^*_R(M, N)$ puts a natural bigrading on the latter. Complexes will be written homologically or cohomologically as needed via $C_n = C^{-n}$, with suspension functor given by $(C[1])_n = C_{n-1}$ and $(C[1])^{n} = C^{n+1}$.

From now on, let $R$ stand for a commutative Noetherian graded connected $k$-algebra of finite Krull dimension $d$. Call $R$ Gorenstein if it is Cohen-Macaulay and its canonical module is invertible, that is $\omega_R \cong R(a)$ for some $a \in \mathbb{Z}$. It is equivalent to require that 
\begin{align*}
{\rm Ext}^i_R(k, R) \cong \begin{cases} 0 & i \neq d\\
										k(-a) & i = d.
						\end{cases}
\end{align*}
If $S = k[x_0, ..., x_n]$ and $R = S/I$ with $I = (f_1, ..., f_c)$ generated by a homogeneous regular sequence in $(x_0, ..., x_n)$, then $R$ is Gorenstein and $\omega_R = R(a)$ with $a = \sum_{j=1}^c |f_j| - \sum_{i=0}^n |x_i|$. 

Call a graded $R$-module $M$ Maximal Cohen-Macaulay (MCM) if $M$ admits an $M$-regular sequence of length $d$. By Local Duality, over a Gorenstein ring this is equivalent to the cohomological vanishing criterion ${\rm Ext}^i_R(M, R) = 0$ for $i > 0$. MCM modules are closed under taking duals $M^* = {\rm Hom}_R(M, R)$, extensions and summands, and are always reflexive. 

Since the canonical module $\omega_R$ always has finite injective dimension equal to $d$, over a Gorenstein ring $R$ any $i^{th}$-syzygy module $syz^i_R(N)$ is MCM for $i \geq d$. It is a theorem of R.-O. Buchweitz that the converse holds:

\begin{thm}[Buchweitz, \cite{BuchweitzManuscript}] The following are equivalent over $R$ Gorenstein:
\begin{enumerate}[i.]
\item M is a MCM module.
\item For any $i \geq 0$ there is an $R$-module $N$ with $M \cong syz_R^i(N)$. 
\end{enumerate}
\end{thm}
The proof proceeds by construction of a complete resolution of $M$: let $P \xrightarrow{\sim} M$ and $Q \xrightarrow{\sim} M^*$ be projective resolutions. Then $M \cong M^{**} \xrightarrow{\sim} Q^*$ is a projective coresolution of $M$ by the cohomological vanishing criterion, and so the composite map $P \xrightarrow{\sim} M \cong M^{**} \xrightarrow{\sim} Q^*$ is a quasi-isomorphism of complexes of projectives
\[
\xymatrix@C1.6pc@R0.5pc{... \ar[r] & P_1 \ar[r] & P_0 \ar[rr] \ar[dr] & & Q_0^* \ar[r] & Q_1^* \ar[r] & ...\\
					   &            &             & M \ar[ur] & 
}
\]
and thus the unbounded complex of projectives $C = {\rm cone}(P \to Q^*)[-1]$ is acyclic. We call any unbounded acyclic complex of projectives whose non-negative truncation resolves $M$ a complete resolution of $M$. Truncating sufficiently far to the right reveals $M$ as a syzygy module $syz^i_R(N)$ for any $i \ge 0$.

For graded $R$-modules $M, N$, denote by $\underline{{\rm Hom}}_{gr R}(M, N) = {\rm Hom}_{gr R}(M, N)/I(M, N)$ the space of stable homomorphisms, where we mod out by the ideal $I = I(M, N)$ of morphisms factoring through projective graded modules. Denote by $\underline{{\rm grmod}} R$ the stable category of graded $R$-modules, and by $\underline{{\rm MCM}}(gr R)$ the subcategory consisting of MCM modules. Note that since $R$ is graded connected and locally finite, ${\rm gr} R$ is Hom finite, and thus so is $\underline{{\rm MCM}}(gr R)$. Let $\mathcal{K}^{\infty}_{ac}({\rm proj} R)$ stand for the homotopy category of unbounded acyclic complexes of finite projective graded $R$-modules. 
\begin{prop}[Buchweitz, \cite{BuchweitzManuscript}] There is an equivalence of categories $\mathcal{K}^{\infty}_{ac}({\rm proj} R) \cong \underline{{\rm MCM}}(gr R)$ given by
\[
C \mapsto {\rm coker}(C_1 \to C_0).
\]
Hence $\underline{{\rm MCM}}(gr R)$ inherits the structure of an (algebraic) triangulated category. In particular the suspension corresponds to taking cosyzygy $M[1] \cong syz^{-1}_R(M) = {\rm coker}(C_0 \to C_{-1})$ for any complete resolution $C$, and its inverse to taking syzygy $M[-1] \cong syz^1_R(M) = {\rm coker}(C_2 \to C_1)$.
\end{prop} 
Complete resolutions are unique up to homotopy, and so the Tate cohomology $\widehat{\rm Ext}^*_{gr R}(M, N) = {\rm H}^*\big({\rm Hom}_{gr R}(C(M), N)\big)$ is well-defined for any module $N$, where $C(M)$ is any complete resolution of $M$. If $N$ is also MCM, the canonical map $C(N) \to N$ induces a bijection ${\rm H}^*\big({\rm Hom}_{gr R}(C(M), C(N)\big) \xrightarrow{\cong} {\rm H}^*\big({\rm Hom}_{gr R}(C(M), N)\big)$, and so $\underline{{\rm Hom}}_{gr R}(M, N[n]) \cong \widehat{\rm Ext}_{gr R}^n(M, N)$ for $n \in \Z$. A complete resolution $C$ is minimal if the differential on $k \otimes_R C$ is trivial. Minimal complete resolutions of MCM modules always exist since $R$ is graded connected, and are unique up to isomorphism.

Note that the distinguished triangles in $\underline{{\rm MCM}}(gr R)$ are the images of short exact sequences in ${\rm MCM}(gr R)$.

We shall make implicit use of another theorem of Buchweitz. Let ${\rm D}_{sg}(gr R) = {\rm D}^b(gr R)/{\rm D}_{{\rm perf}}(gr R)$ stand for the graded singularity category of $R$, where we take the Verdier quotient by the perfect complexes. Then
\begin{prop}[Buchweitz, \cite{BuchweitzManuscript}] There is an equivalence of triangulated categories ${\rm D}_{sg}(gr R) \cong \underline{{\rm MCM}}(gr R)$.
\end{prop}
The composite functor ${\rm D}^b(gr R) \to \underline{{\rm MCM}}(gr R)$ can be described as follows: let ${\rm F}_*$ be a complex with bounded cohomology, resolved by a lower bounded complex of projectives $P_* \xrightarrow{\sim} {\rm F}_*$. Then the tail of $P_*$ is exact, and so 
\[
{\rm coker}(P[-n]_1 \xrightarrow{\partial} P[-n]_0) =: M
\]
is MCM for $n \gg 0$. The image of ${\rm F}_*$ in $\underline{{\rm MCM}}(gr R)$ is then given by $M[n] = {\rm syz}_R^{-n}(M)$. We call this the MCM approximation  (or stabilization) of ${\rm F}_*$, denoted ${\rm F}^{st}_*$. 

Note that the induced map ${\rm st}: {\rm Ext}_{gr R}^n(M, N) \to \widehat{\rm Ext}_{gr R}^n(M^{st}, N^{st})$ for modules $M, N$ is a bijection whenever $n \geq d$, or whenever $n \geq 1$ if $M$ is itself MCM.

Finally, call an additive category Krull-Schmidt if objects have finite decompositions into indecomposables and the endomorphism rings of indecomposables are local (\cite{MR3431480}). The following was observed by Iyama and Takahashi, where it follows from idempotent lifting in artinian rings.

\begin{lem}[Iyama-Takahashi, \cite{MR3063907}]\label{krull-schmidt}
The categories ${\rm MCM}(gr R)$ and $\underline{{\rm MCM}}(gr R))$ are Krull-Schmidt.
\end{lem} 

By the Krull-Remak-Schmidt Theorem it follows that each graded MCM module has an essentially unique decomposition into indecomposable modules, and that MCM modules with no free summands are stably isomorphic if and only if they are isomorphic. Hence the classification of MCM modules reduces to the classification of indecomposable stable modules.

\subsection*{Auslander-Reiten-Serre duality and almost-split sequences}\label{serreduality}
For the rest of this section assume that $R$ is Gorenstein. We briefly review standard background that is common to algebraic geometry, commutative algebra and representation theory of artin algebras.

\begin{defn}
Let $\mathcal{T}$ be a triangulated Hom-finite $k$-linear category. A Serre functor for $\mathcal{T}$ is an exact autoequivalence $\mathbb{S}: \mathcal{T} \to \mathcal{T}$ equipped with natural isomorphisms
\[
{\rm Hom}_{\mathcal{T}}(X, \mathbb{S}(Y)) \cong {\rm D}{\rm Hom}_{\mathcal{T}}(Y, X)
\]
where ${\rm D} = {\rm Hom}_k(- ,k)$ denotes the $k$-dual. When they exist, Serre functors are unique up to isomorphism.
\end{defn}

If $X$ is a smooth projective variety over $k$, the functor $\mathbb{S} = - \otimes \omega_X  [\textup{dim } X]$ is a Serre functor for ${\rm D}^b(X)$. It is quite remarkable that this extends to the stable category of MCM modules.

\begin{prop}[Auslander \cite{MR842486}, \cite{MR3063907}] Assume that $R$ has isolated singularities. Then $\mathbb{S} = - \otimes \omega_R \ [\textup{dim } R - 1]$ is a Serre functor for $\underline{{\rm MCM}}(gr R)$.
\end{prop}

Serre functors were independently discovered by Auslander and Reiten in the guise of the  translate $\tau$, introduced in the context of stable module categories, see \cite{MR1476671}. In \cite{MR1887637}, Reiten-Van den Bergh studied $\tau$ in the context of a Krull-Schmidt Hom-finite $k$-linear triangulated category $\mathcal{T}$. Let $\xi: X \to Y \to Z \xrightarrow{h} X[1]$  be a distinguished triangle in $\mathcal{T}$. We call $\xi$ an almost-split triangle if
\begin{enumerate}[i.]
\item X and Z are indecomposable,
\item $h \neq 0$,
\item if $W$ is indecomposable, then for every non-isomorphism $t: W \to Z$ we have $ht = 0$. 
\end{enumerate}

\begin{prop}[Reiten-Van den Bergh, \cite{MR1887637}] The following are equivalent:
\begin{enumerate}[1.]
\item Each indecomposable $Z$ of $\mathcal{T}$ sits inside an almost-split triangle, say $\tau Z \to Y \to Z \xrightarrow{h} \tau Z [1]$,
\item $\mathcal{T}$ admits a Serre functor $\mathbb{S}$.
\end{enumerate}
In this case $\tau = \mathbb{S}\circ [-1]$ and the map $Z \xrightarrow{h} \tau Z [1] = \mathbb{S}(Z)$ classifying the extension is Serre dual to the trace map ${\rm End}(Z) \to {\rm End}(Z)/{\rm rad} = k$.
\end{prop}

Now for $\mathcal{T}$ as above, the Auslander-Reiten quiver $\Gamma(\mathcal{T})$ is the quiver whose vertices are the isomorphism classes of indecomposables of $\mathcal{T}$ and arrows taken from a basis for ${\rm Irr}(X, Y)$, the space of equivalence classes of irreducible maps between indecomposables (see e.g. \cite{MR935124} for details). The Auslander-Reiten quiver $\Gamma(\mathcal{T})$ is related to almost-split triangles as follows.
\begin{prop}[Ringel, {\cite[4.8]{MR935124}}]\label{Ringel} Let $X, M, Z$ be indecomposable objects in $\mathcal{T}$ and $X \to Y \to Z \to X[1]$ almost-split, with $Y = \bigoplus_{i=1}^r Y_{i}^{\oplus d_i}$ decomposed into pairwise non-isomorphic indecomposables. Then ${\rm Irr}(X, M) \neq 0$ if and only if $M \cong Y_i$ for some i, in which case $d_i = {\rm dim}\ {\rm Irr}(X, Y_i)$. 
\end{prop}

We need one more example of a triangulated category $\mathcal{T}$ with a Serre functor for our purposes. 

\begin{prop}[Happel, \cite{MR935124}] Let $A$ be a finite-dimensional $k$-algebra of finite global dimension. Then
${\rm D}^b(A)$ is a Krull-Schmidt Hom-finite triangulated category, and the functor $\mathbb{S} = \omega_A \otimes^{\mathbb{L}}_A -$ is a Serre functor where $\omega_A = DA$, with inverse $\mathbb{S}^{-1} = {\rm RHom}_A(\omega_A, -)$. Hence the category ${\rm D}^b(A)$ admits almost-split triangles.
\end{prop}

We will only care about the hereditary case gldim $A = 1$. For $M$ is an $A$-module without projective summand (resp. injective summand), we have $\tau^{-1} M = \mathbb{S}^{-1}(M)[1] \cong {\rm Ext}^1_A(DA, M)$ (resp. $\tau M = \mathbb{S}(M)[-1] \cong {\rm Tor}^A_1(DA, M)$), and so $\tau$ as defined on ${\rm D}^b(A)$ agrees with the usual definition as in \cite{MR2197389}, \cite{MR1476671}. Moreover Happel has shown in \cite[4.7]{MR935124} that in this case, an almost-split short exact sequence of $A$-modules $0 \to X \to Y \to Z \to 0$ gives rise to an almost-split triangle $X \to Y \to Z \to X[1]$ in ${\rm D}^b(A)$. We will implicitly interpret any needed result as living in ${\rm D}^b(A)$.

\subsection{The Four Subspace problem}\label{foursubspace}

The Four Subspace problem asks for a classification of normal forms of possible subspace configurations $(W; V_1, V_2, V_3, V_4)$ with $V_i \subseteq W$, up to change of $W$-basis. As we can take direct sums of such data, one asks for a classification of indecomposable subspace configurations. Let $Q$ be the $\widetilde{D}_4$ quiver oriented as follows
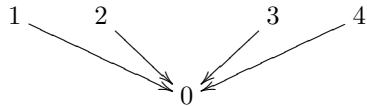
\[
\xymatrix@R1.5pc@C1.8pc{ 1 \ar[drr] & 2 \ar[dr] & & 3 \ar[dl] & 4 \ar[dll] \\
					 & 		   & 0 & 
}
\]
A representation $M$ of $Q$ is given by linear maps
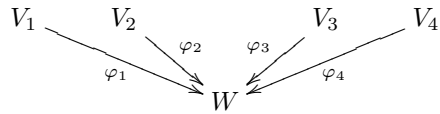
\[
\xymatrix@R1.5pc@C1.8pc{ V_1 \ar_{\varphi_1}[drr] & V_2 \ar^{\varphi_2}[dr] & & V_3 \ar_{\varphi_3}[dl] & V_4 \ar^{\varphi_4}[dll] \\
				 & 	   & W & 
 }
\]
and subspace configurations correspond to representations of $Q$ with all $\varphi_i$ injective. If a representation $M$ is indecomposable then either all $\varphi_i$ are injective or $M$ is one of the simple representations $S(i)$ for $i = 1,2,3,4$:
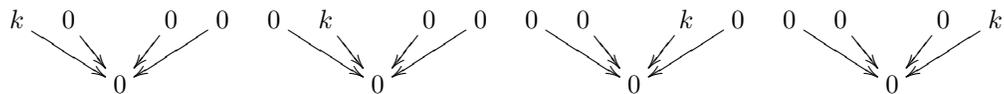
\[
\xymatrix@R1pc@C.7pc{ k \ar[drr] & 0 \ar[dr] & & 0 \ar[dl] & 0 \ar[dll] & 0 \ar[drr] & k \ar[dr] & & 0 \ar[dl] & 0 \ar[dll] & 0 \ar[drr] & 0 \ar[dr] & & k \ar[dl] & 0 \ar[dll] & 0 \ar[drr] & 0 \ar[dr] & & 0 \ar[dl] & k \ar[dll]\\
		  & & 0 &  & & & & 0 & & & & & 0 & & & & & 0 & &
}
\]
Representations of $Q$ correspond naturally to left modules over the path algebra $A = kQ$. The algebra $A$ is hereditary, and since the obstructions to formality of a complex of modules $M^{\bullet}$ live in ${\rm Ext}_A^{n, 1-n}({\rm H}^*(M^{\bullet}), {\rm H}^*(M^{\bullet}))$ for $n \geq 2$, complexes in ${\rm D}(A)$ are formal and the indecomposables are of the form $M[n]$ for $M$ an indecomposable $A$-module and $n \in \mathbb{Z}$ (see \cite[Chp. 4]{MR2385175} or \cite{MR935124} for a direct proof of formality).

Nazarova and Gelfand-Ponomarev gave the first classification of indecomposable subspace configurations in \cite{MR0223352}, \cite{MR0338018}, \cite{MR0357428}. Auslander-Reiten theory can be used to give an efficient treatment of the classification. We quote the needed results from \cite{MR1476671} and \cite{MR935124}. 

Each indecomposable complex lives in one of two types of components of the Auslander-Reiten quiver of ${\rm D}^{\sf b}(A)$:

\underline{Preprojective-preinjective components}: 

First, the indecomposable projective modules $P(i)$ for $i = 0, 1,2,3,4$ are given by
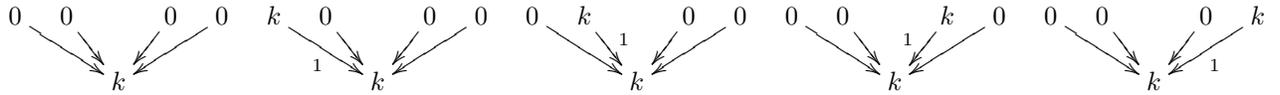
\[
\xymatrix@R1pc@C.7pc{0 \ar[drr] & 0 \ar[dr] & & 0 \ar[dl] & 0 \ar[dll] & k \ar_{1}[drr] & 0 \ar[dr] & & 0 \ar[dl] & 0 \ar[dll] & 0 \ar[drr] & k \ar^{1}[dr] & & 0 \ar[dl] & 0 \ar[dll] & 0 \ar[drr] & 0 \ar[dr] & & k \ar_{1}[dl] & 0 \ar[dll] & 0 \ar[drr] & 0 \ar[dr] & & 0 \ar[dl] & k \ar^{1}[dll]\\
			&	& k &  & & & & k &  & & & & k & & & & & k & & & & & k &
}
\]
and the indecomposable injective modules $I(i)$ for $i = 0,1,2,3,4$ by
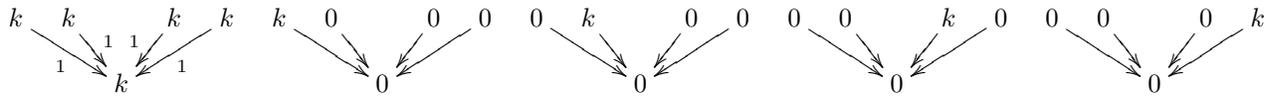
\[
\xymatrix@R1pc@C.7pc{k \ar_{1}[drr] & k \ar^{1}[dr] & & k \ar_{1}[dl] & k \ar^{1}[dll] & k \ar[drr] & 0 \ar[dr] & & 0 \ar[dl] & 0 \ar[dll] & 0 \ar[drr] & k \ar[dr] & & 0 \ar[dl] & 0 \ar[dll] & 0 \ar[drr] & 0 \ar[dr] & & k \ar[dl] & 0 \ar[dll] & 0 \ar[drr] & 0 \ar[dr] & & 0 \ar[dl] & k \ar[dll]\\
			&	& k &  & & & & 0 &  & & & & 0 & & & & & 0 & & & & & 0 &
}
\]
Let $\mathcal{PI}$ stand for the connected component of $\Gamma({\rm D}^b(A))$ containing $P(i)$ for $i=0,1,2,3,4$. According to \cite[I.5]{MR935124} this takes the form
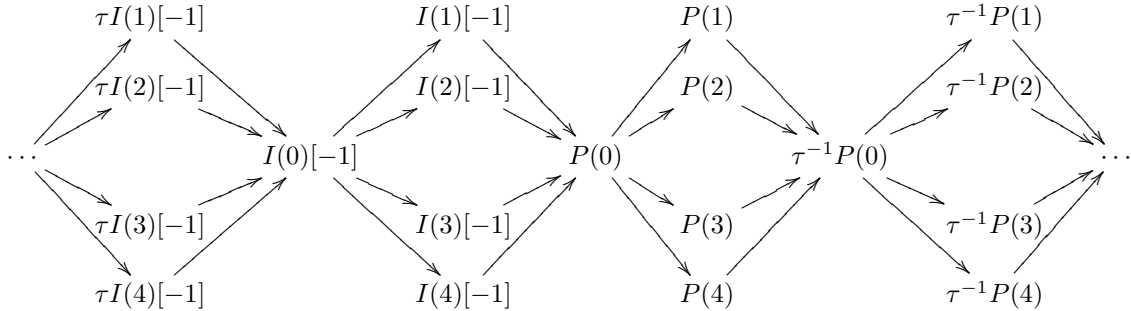
\[
\xymatrix@R.8pc@C1.3pc{
& \tau I(1)[-1] \ar[ddr] && I(1)[-1] \ar[ddr] && P(1) \ar[ddr] && \tau^{-1}P(1) \ar[ddr]\\
& \tau I(2)[-1] \ar[dr] && I(2)[-1] \ar[dr] && P(2)  \ar[dr] &&  \tau^{-1}P(2)  \ar[dr]\\
\cdots \ar[uur] \ar[ur] \ar[dr] \ar[ddr] && I(0)[-1] \ar[uur] \ar[ur] \ar[dr] \ar[ddr] && P(0)\ar[uur] \ar[ur] \ar[dr] \ar[ddr] && \tau^{-1}P(0)\ar[uur] \ar[ur] \ar[dr] \ar[ddr] && \cdots \\
& \tau I(3)[-1] \ar[ur] && I(3)[-1] \ar[ur] && P(3) \ar[ur] && \tau^{-1}P(3) \ar[ur]\\
& \tau I(4)[-1] \ar[uur] && I(4)[-1] \ar[uur] && P(4) \ar[uur] && \tau^{-1}P(4) \ar[uur]\\
}
\]

We let $\mathcal{PI}[n]$ be the $n^{th}$ suspension of $\mathcal{PI}$ for $n \in \Z$. These components contain precisely all modules of the form $\tau^m P(i)[n]$ for $m, n \in \Z$, since we have $I(i)[-1] \cong \tau P(i)$.

\underline{Regular components}:

Modules whose indecomposable summands are not of the previous type are said to be regular. Let us recall a well-known property of regular modules.
\begin{prop}\label{serial}
The category $\mathcal{R}(Q)$ of regular modules is an abelian subcategory of $kQ$-mod, and is serial in that every indecomposable has a unique composition series with simple regular factors.
\end{prop}

We classify the indecomposable regular modules as follows. First, consider representations of the form
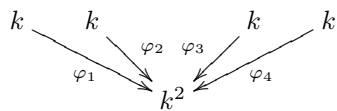
\[
\xymatrix@R1.4pc@C1.4pc{ k \ar_{\varphi_1}[drr] & k \ar^{\varphi_2}[dr] & & k \ar_{\varphi_3}[dl] & k \ar^{\varphi_4}[dll] \\
				 & 	   & k^2 & 
 }
\]
with each $\varphi_i \neq 0$, which represent $4$-tuples of lines in $k^2$ through the origin, or $4$-tuples of points on $\mathbb{P}^1$. Assume that the first $3$ points are distinct, so that up to projective transformation these points can be taken as \mbox{$\big( [0:1], [1:1], [1:0], [t_0:t_1] \big)$} with $t = [t_0: t_1] \in \mathbb{P}^1$. Let $R_{t}$ be given by
\[
\xymatrix@R2.0pc@C2.3pc{ k \ar_{\begin{bsmallmatrix} 0 \\1 \end{bsmallmatrix}}[drr] & k \ar^{\begin{bsmallmatrix}1\\1\end{bsmallmatrix}}[dr] & & k \ar_{\begin{bsmallmatrix}1 \\ 0 \end{bsmallmatrix}}[dl] & k \ar^{\begin{bsmallmatrix}t_0 \\ t_1 \end{bsmallmatrix}}[dll] \\
				 & 	   & k^2 & 
 }
\]
We call the corresponding family of representations $\mathcal{R} = \{ R_{t} \}_{t \in \mathbb{P}^1}$ the {\bf fundamental family}. 

\begin{prop} We have ${\rm End}(R_t) = k$, hence the representation $R_t$ is indecomposable for every $t \in \mathbb{P}^1$.
\end{prop}

Let $\mathcal{T}_{t}^Q$ stand for the Auslander-Reiten component containing $R_{t}$. We will describe them in detail.

 \underline{Case $t \in \mathbb{P}^1 \setminus \{0, 1, \infty \}$}:
 
 In this case all lines are distinct. We first describe a family of representations $R_{t}\langle r \rangle$ for $r \geq 1$.
 \begin{prop}\label{R_t} The following holds.
 \begin{enumerate}[i.]
 \item There are unique non-trivial extensions
 \[
0 \to R_{t} \to R_{t}\langle 2 \rangle \to R_{t} \to 0
\]
\[
0 \to R_{t}\langle r \rangle \to R_{t} \langle r+1 \rangle \to R_{t} \to 0
\]

and $R_{t}\langle r \rangle$ are indecomposable for all $r \geq 2$.
\item Setting $R_{t}\langle 1 \rangle = R_{t}$, we have $\tau \big( R_{t} \langle r \rangle \big) \cong R_{t} \langle r \rangle$ for all $r \geq 1$.
 \end{enumerate}
 \end{prop}
 
Following Prop \ref{serial}, the module $R_t$ is the unique simple regular module in its component and $r$ denotes the length of the uniserial module $R_t\langle r \rangle$. The component $\mathcal{T}^{Q}_{t}$ looks like:
 \[
 \xymatrix@R.8pc@C.8pc{
 \vdots \ar[dr] \ar@{-}[dd] && \vdots \ar@{-}[dd]\\
 & R_{t}\langle 4 \rangle \ar[dr] \ar[ur] & \\
 R_{t}\langle 3 \rangle \ar[dr] \ar[ur] \ar@{.}[rr] \ar@{-}[dd] && R_{t}\langle 3 \rangle \ar@{-}[dd] \\
 & R_{t} \langle 2 \rangle \ar[dr] \ar[ur] & \\
 R_{t} \ar[ur] \ar@{.}[rr] && R_{t}
 }
 \]
 
 where the edges are identified, and $\mathcal{T}_{t}^Q$ forms a tube of rank one.
 
\underline{Case $t = 0, 1, \infty$}:

Consider again  $4$-tuples of lines given by $\big( \varphi_1, \varphi_2, \varphi_3, \varphi_4 \big)$ as above, and now assume that exactly two of the lines collide. This yields a partition of $\{\varphi_1, \varphi_2, \varphi_3, \varphi_4 \} = \{\varphi_i, \varphi_j \} \coprod \{\varphi_p, \varphi_q \}$ where $\varphi_i = \varphi_j$ and $\varphi_p \neq \varphi_q$. Keeping track of ordering, there are $\binom{4}{2} = 6$ such partitions. Let us define corresponding representations as
\begin{align*}
R_0^{+}: \big( [0\ 1]^t, [1\ 1]^t, [1\ 0]^t, [0\ 1]^t \big) && R_0^{-}: \big([1\ 1]^t, [0\ 1]^t, [0\ 1]^t, [1\ 0]^t \big)\\
R_1^{+}: \big( [0\ 1]^t, [1\ 1]^t, [1\ 0]^t, [1\ 1]^t \big) && R_1^{-}: \big( [1\ 1]^t, [0\ 1]^t, [1\ 1]^t, [1\ 0]^t \big)\\
R_{\infty}^{+}: \big( [0\ 1]^t, [1\ 1]^t, [1\ 0]^t, [1\ 0]^t \big) && R_{\infty}^{-}: \big( [1\ 0]^t, [1\ 0]^t, [0\ 1]^t, [0\ 1]^t \big)
\end{align*}
We have chosen this normalization with the following properties in mind:
\begin{enumerate}[i.]
\item For $t = 0, 1, \infty$, we have $R_{t} = R_{t}^{+}$.
\item The involution $R_{t}^{+} \leftrightarrow R_{t}^{-}$ corresponds to interchanging $\{\varphi_i, \varphi_j\}$ and $\{\varphi_p, \varphi_q \}$.
\end{enumerate}

\begin{prop} The following hold.
\begin{enumerate}[i.]
\item The representations $R_{t}^{\pm}$ are pairwise non-isomorphic and indecomposable.
\item We have $\tau \big( R_{t}^{+} \big) \cong R_{t}^{-}$ and $\tau \big( R_{t}^{-} \big) \cong R_{t}^{+}$.
\end{enumerate}
\end{prop}
We need yet more indecomposable modules to describe the components $\mathcal{T}_{t}^Q$ for $t = 0, 1, \infty$. Let us introduce
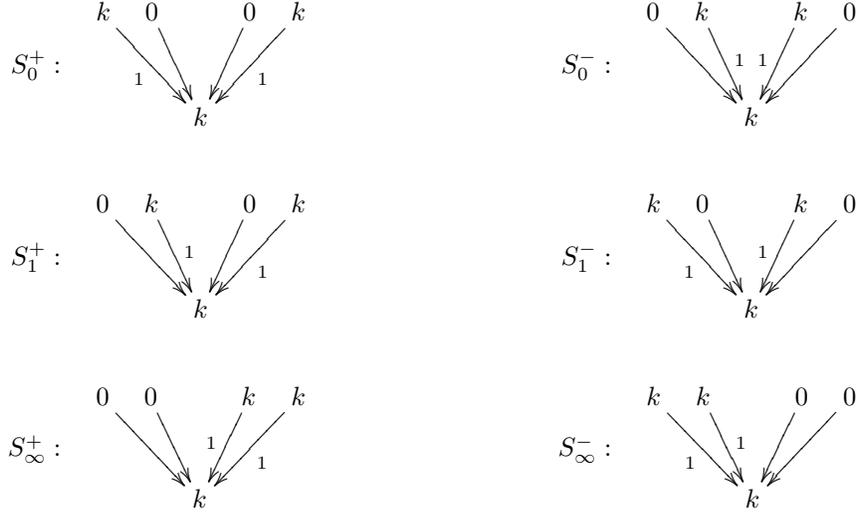
\begin{align*}\label{Spm}
\xymatrix@R.4pc@C.6pc{ & k \ar_{1}[ddrr] & 0 \ar[ddr] & & 0 \ar[ddl] & k \ar^{1}[ddll] \\
 S_0^+: & \\
			&	 & 	   & k & 
 }  
    &&  \xymatrix@R.4pc@C.6pc{& 0 \ar[ddrr] & k \ar^{1}[ddr] & & k \ar_{1}[ddl] & 0 \ar[ddll] \\
    S^-_0: & \\
		&		& 	   & k & 
 }  
\\
\\
     \xymatrix@R.4pc@C.6pc{& 0 \ar[ddrr] & k \ar^1[ddr] & & 0 \ar[ddl] & k \ar^1[ddll] \\
     S_1^+: & \\
				& & 	   & k & 
 }  
    &&  \xymatrix@R.4pc@C.6pc{ & k \ar_1[ddrr] & 0 \ar[ddr] & & k \ar_1[ddl] & 0 \ar[ddll] \\
    S_1^-: & \\
		&	 & 	   & k & 
 }   
\\
\\
     \xymatrix@R.4pc@C.6pc{& 0 \ar[ddrr] & 0 \ar[ddr] & & k \ar_1[ddl] & k \ar^1[ddll] \\
     S_{\infty}^+: & \\
				& & 	   & k & 
 }  
    && \xymatrix@R.4pc@C.6pc{& k \ar_1[ddrr] & k \ar^1[ddr] & & 0 \ar[ddl] & 0 \ar[ddll] \\
    S_{\infty}^-: & \\
				& & 	   & k & 
 }   
\end{align*}

\begin{prop}\label{S_t} The following hold.
\begin{enumerate}[i.]
\item We have $\tau \big( S_{t}^{+} \big) \cong S_{t}^{-}$ and $\tau \big( S_{t}^{-} \big) \cong S_{t}^{+}$.
\item There are short exact sequences
\[
0 \to S_{t}^{+} \to R_{t}^{+} \to S_{t}^{-} \to 0
\]
\[
0 \to S_{t}^{-} \to R_{t}^{-} \to S_{t}^{+} \to 0
\]
\end{enumerate}
\end{prop}

We will not explicitly write down further indecomposables, but rather characterize them homologically as this better suits our needs. 
\begin{prop} There are unique non-trivial extensions
\[
0 \to R_{t}^{+} \to  S_{t}\langle 3 \rangle^{+} \to S_{t}^{-} \to 0
\]
\[
0 \to R_{t}^{-} \to S_{t}\langle 3 \rangle^{-} \to S_{t}^{+} \to 0
\]

and the middle term $S_{t}\langle 3 \rangle^{\pm}$ is indecomposable. Similarly, there are unique non-trivial extensions
\[
0 \to S_{t}\langle r \rangle^{+} \to S_{t}\langle r+1 \rangle^{+} \to S_{t}^{-} \to 0
\]
\[
0 \to S_{t}\langle r \rangle^{-} \to S_{t}\langle r+1 \rangle^{-} \to S_{t}^{+} \to 0
\]

with indecomposable middle terms for all $r \geq 3$. We let $S_{t}\langle 1 \rangle^{\pm} := S_{t}^{\pm}$ and $S_{t}\langle 2 \rangle^{\pm} := R_{t}^{\pm}$ for consistency.
\end{prop}

Again, $S_t^{\pm}$ are the only simple regular modules in  $\mathcal{T}_t^Q$ and $r$ denotes the length of the uniserial module $S_t \langle r \rangle^{\pm}$. Finally, the component $\mathcal{T}_{t}^Q$ for $t = 0, 1, \infty$ is given by

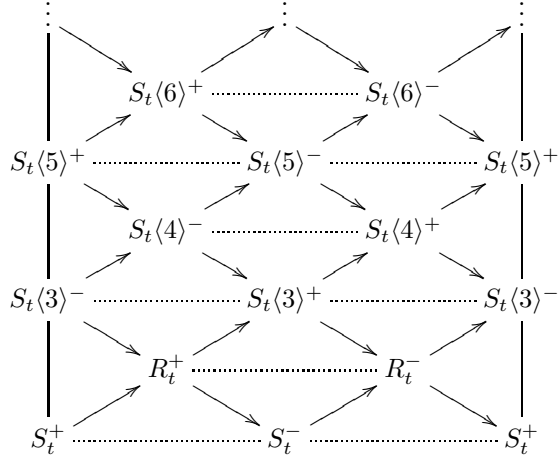
\[
\xymatrix@R.8pc@C.8pc{
\vdots \ar[dr] \ar@{-}[dd] && \vdots \ar[dr] && \vdots \ar@{-}[dd] \\
& S_{t}\langle 6 \rangle^{+} \ar[dr] \ar[ur] \ar@{.}[rr] && S_{t}\langle 6 \rangle^{-} \ar[dr] \ar[ur] & \\
S_{t}\langle 5 \rangle^{+} \ar[dr] \ar[ur] \ar@{-}[dd] \ar@{.}[rr] && S_{t}\langle 5 \rangle^{-} \ar[dr] \ar[ur] \ar@{.}[rr] && S_{t}\langle 5 \rangle^{+} \ar@{-}[dd] \\
& S_{t}\langle 4 \rangle^{-} \ar[dr] \ar@{.}[rr] \ar[ur] && S_{t}\langle 4 \rangle^{+} \ar[dr] \ar[ur] \\
S_{t}\langle 3 \rangle^{-} \ar@{-}[dd] \ar@{.}[rr] \ar[dr] \ar[ur] && S_{t}\langle 3 \rangle^{+} \ar[dr] \ar[ur] \ar@{.}[rr] && S_{t}\langle 3 \rangle^{-} \ar@{-}[dd] \\
& R_{t}^{+} \ar[dr] \ar[ur] \ar@{.}[rr] && R_{t}^{-} \ar[dr] \ar[ur] \\
S_{t}^{+} \ar[ur] \ar@{.}[rr] && S_{t}^{-} \ar[ur] \ar@{.}[rr] && S_{t}^{+}
}
\]

where we identify the left and right edges, forming a tube of rank two.

Now let $t \in \mathbb{P}^1$ again. As before we let $\mathcal{T}_{t}^Q[n]$ stand for the $n^{th}$ suspension of $\mathcal{T}_{t}^Q$. We note for the record:
\begin{prop}
The components $\mathcal{T}_{t}^Q$ are orthogonal for distinct $t \in \mathbb{P}^1$.
\end{prop}

\begin{thm}\label{ARquiverQ} The Auslander-Reiten quiver of ${\rm D}^b(A)$ is given by the union of disjoint components
\[
\Gamma({\rm D}^b(A)) = \bigg(\bigcup_{n \in \Z} \mathcal{PI}[n] \bigg) \cup \bigg(\bigcup_{t \in \mathbb{P}^1, n \in \Z} \mathcal{T}_{t}^Q[n]\bigg).
\]
In particular this yields a full classification of indecomposables in ${\rm D}^b(A)$.
\end{thm}

\section{List of indecomposable graded MCM modules}

From now on, we assume that char $k \neq 2$ and we let $X$ be a set of $4$ distinct reduced points in $\mathbb{P}^2$ with no $3$ on a line, and let $R = k[X]$ be the homogeneous coordinate ring. Under these assumptions it is well-known that $X$ is the intersection of two quadrics, that there is exactly one pencil of quadrics through $X$ and that this pencil contains precisely $3$ singular elements. Furthermore, we can assume that $X$ is given in some coordinates by the points $[\pm 1: \pm 1: \pm 1]$, and that the singular quadrics are given by
\begin{align*}
& Q_0 = x^2 - y^2\\
& Q_1 = x^2 - z^2\\
& Q_\infty = y^2 - z^2
\end{align*}
with general quadric given by $\{ Q_t = 0\}$ with $Q_t = t_0Q_\infty + t_1 Q_0$, $t = [t_0:t_1] \in \mathbb{P}^1$. 

The complete intersection $R = k[X] = k[x,y,z]/(Q_0, Q_{\infty})$ is a graded Gorenstein isolated singularity with $a = 1$ and so $\omega_R = R(1)$. Note that the Serre functor on $\underline{{\rm MCM}}(gr R)$ is given by $\mathbb{S} = (1)$ and the translate by $\tau = (1)[-1] = syz_R^1(-)(1)$. 

The aim of this section is to describe the minimal complete resolution of MCM modules over $R$. Following Buchweitz-Pham-Roberts \cite{BuchweitzCI}, complete resolutions over complete intersections take a special form, and we begin by reviewing their results. 

\subsection*{Complete resolutions over complete intersections} Let $R = S/(f_1, ..., f_c)$ be a homogeneous complete intersection with $S = k[x_0, ..., x_n]$, which can be resolved and thus replaced by the Koszul complex $\mathbb{K} = \left( \Lambda_S(df_1, ..., df_c), \delta \right) \xrightarrow{\sim} R$, with $|df_i| = (1, |f_i|)$. Let $N$ be an $R$-module. As shown by Avramov and Buchweitz in \cite{MR1774757}, $N$ has a graded free $S$-resolution $(E, \partial) \xrightarrow{\sim} N$ which admits a bigraded dg module structure over $\mathbb{K}$.

Now for $B$ any commutative ring, let $\Gamma = \Gamma_{B}(\eta_1, ..., \eta_c)$ be the free divided power $\Z$-algebra on $c$ generators with $|\eta_j| = (2, |f_j|)$. Following \cite{MR1774757}, for $\overline{E} = R \otimes_S E$ the complex $\left(\Gamma_R \otimes_R \overline{E}, \partial \right)$ is a free $R$-resolution of $N$, where $\partial(\underline{\eta} \otimes m) = \partial_\tau(\underline{\eta} \otimes m) +  \underline{\eta} \otimes \partial(m)$, with $\underline{\eta}$ a monomial in $\Gamma_R$ and where $\partial_\tau(\eta_j \otimes m) = 1 \otimes df_j \cdot m$, extended as an $R$-linear derivation over $\Gamma_R$.

We will associate a complete resolution to $\left(\Gamma_R \otimes_R \overline{E} , \partial \right)$ after a digression through generalities. For now, let $\Lambda = \Lambda_S(V)$ be a general exterior algebra over $S$, and consider $\Lambda$ as an Hopf $S$-algebra. Then $\Lambda$ is finite and projective over $S$ and the bimodule $\omega = {\rm Hom}_S(\Lambda, S)$ is invertible, since dualising a choice of top form $\Omega$ yields an isomorphism $\omega \cong \Lambda \otimes_S {\rm det}(V)^{\vee}$, where $(-)^{\vee}$ is the $S$-dual. Cartan has shown that the minimal resolution $\Lambda \otimes_S \Gamma \xrightarrow{\sim} S$ of $S$ over $\Lambda$ is given by
\[
\cdots \xrightarrow{d} \Lambda \otimes_S \Gamma^2 \xrightarrow{d} \Lambda \otimes_S \Gamma^1 \to \Lambda \to 0
\]
where here $\Gamma = \Gamma_S(V)$ is the divided power algebra on $V$ and $d$ is the $\Lambda$-linear derivation extending $\tau: \Gamma \twoheadrightarrow V \hookrightarrow \Lambda$ (see \cite{MR1648664}). Dualising\footnote{Dualising turns left $\Lambda$-modules into right $\Lambda$-modules, which we implicitly turn back into a left module by pulling back through the antipode.} gives rise to a projective coresolution $S \cong S^{\vee} \xrightarrow{\sim} \omega \otimes_S T_S$ where $T_S = {\rm Sym}_S(V^{\vee}) = \Gamma^{\vee}_S$, and taking the (shifted) cone of the quasi-isomorphism $\Lambda \otimes_S \Gamma_S \xrightarrow{\sim} S \cong S^{\vee} \xrightarrow{\sim} \omega \otimes_S T_S$ gives a complete resolution $C(S)$ of the form
\[
\xymatrix@C1.6pc@R0.5pc{ \cdots \ar[r] & \Lambda \otimes_S \Gamma^1_S \ar[r] & \Lambda \ar[rr] \ar[dr] & & \omega \ar[r] & \omega \otimes_S T^1_S \ar[r] & \cdots \\
					   &            &             & S \ar[ur] & 
}
\]
with differential given by
\begin{equation}
\begin{cases}
    d(a \otimes \underline{\eta}) & = \sum_{i} a\cdot \tau(\eta_i) \otimes \partial_{\eta_i} \cdot \underline{\eta}\\
    d(a) & = a\cdot \Omega \otimes \Omega^{\vee} \\
    d(a \otimes \Omega^{\vee} \otimes \partial_{\underline{\eta}}) & = \sum_{i} a\cdot \tau(\eta_i) \otimes \Omega^{\vee} \otimes \partial_{\eta_i} \partial_{\underline{\eta}} 
\end{cases}
\end{equation}
 where $\eta_i$ runs over a basis of $\Gamma^1$, $\underline{\eta} \neq 1$ and $\partial_{\underline{\eta}}$ are general monomials in $\Gamma$ and $T$ with $T$ acting on $\Gamma$ by derivation, and where $\Omega \otimes \Omega^{\vee} \in \omega \cong \Lambda \otimes_S {\rm det}(V)^{\vee}$ is the canonical socle element. Note that this complete resolution is also minimal.

Next, for a graded $\Lambda$-module $E$ which is finite and projective over $S$, using the Hopf structure on $\Lambda$ we can make use of the exterior tensor $M \boxtimes_S E$. Note that we have a canonical isomorphism\footnote{That this is an isomorphism is immediate for $S = E$, which in turn implies it for $E$ finite and projective over $S$.} of left modules $\Lambda \otimes_S E \xrightarrow{\cong} \Lambda \boxtimes_S E$ given by $\Lambda \otimes_S E \xrightarrow{\Delta \otimes 1} \Lambda \boxtimes_S \Lambda \otimes _S E \to \Lambda \boxtimes_S E$. We then form $C(E) := C(S) \boxtimes_S E$ by twisting the differential with the Hopf structure (using Sweedler notation)
\begin{equation}
\begin{cases}
    d(a \otimes \underline{\eta} \otimes m) & = \sum_{i} a\cdot \tau(\eta_i)_{(1)} \otimes (\partial_{\eta_i} \cdot \underline{\eta}) \otimes \tau(\eta_i)_{(2)} \cdot m\\
    d(a \otimes m) & = a\cdot \Omega_{(1)} \otimes \Omega^{\vee} \otimes \Omega_{(2)} \cdot m \\
    d(a \otimes \Omega^{\vee} \otimes \partial_{\underline{\eta}} \otimes m) & = \sum_{i} a\cdot \tau(\eta_i)_{(1)} \otimes \Omega^{\vee} \otimes \partial_{\eta_i} \partial_{\underline{\eta}} \otimes \tau(\eta_i)_{(2)} \cdot m.
\end{cases}
\end{equation}

Equivalently, $C(E)$ can be described as follows: first we form a complete resolution of the bimodule $\Lambda$ via base change over the diagonal map $\Delta: \Lambda \to \Lambda \otimes_S \Lambda$ to form $C(\Lambda) := C(S) \otimes_{\Lambda} (\Lambda \boxtimes_S \Lambda)$, which stays acyclic as $\Lambda \boxtimes_S \Lambda$ is projective. Then $C(E) = C(S) \boxtimes_S E = C(\Lambda) \otimes_{\Lambda} E$ remains acyclic, and we have shown:

\begin{prop}[Buchweitz-Pham-Roberts, \cite{BuchweitzCI}]\label{BPR1} There is a complete resolution of $E$ over $\Lambda$ of the form
\[
\xymatrix@C1.6pc@R0.5pc{ ... \ar[r] & \big(\Lambda \otimes_S \Gamma^1_S\big) \boxtimes_S E  \ar[r] & \Lambda \boxtimes_S E \ar[rr] \ar[dr] & & \omega \boxtimes_S E \ar[r] &  \big(\omega \otimes_S T_S^1\big) \boxtimes_S E  \ar[r] & ...\\
					   &            &             & E \ar[ur] & 
}
\]
\end{prop}

Finally, let $\Lambda = \left(\Lambda_S(df_1, ..., df_c), \delta \right) \xrightarrow{\sim} R$ and $E = (E, \partial) \xrightarrow{\sim} N$ as above. Let $(C(E), d) = (C(S)\boxtimes_S E, d\otimes 1)$. Then $\delta$ and $\partial$ induce a new differential $\tilde{d} := d\otimes 1 + \pm \delta \otimes 1 + 1 \otimes \partial$ on $C(E)$, and the finite filtration on $C(E)$ by total $\Lambda$ and $E$ degrees is lowered by $\pm \delta \otimes 1 + 1 \otimes \partial$, hence a standard spectral sequence argument shows that $(C(E), \tilde{d})$ remains acyclic, now as a dg $\Lambda$-module. Let $C(E)_{\leq -n-1}$ be the dg submodule of elements with $T_S$-degree $\geq n$, corresponding to the right tail in Prop. \ref{BPR1} starting at $\omega \otimes_S T^n_S \otimes_S E$. The dg module $C(E)_{\geq -n} := C(E)/C(E)_{\leq -n-1}$ is easily seen to be semi-free, and so the quasi-isomorphism $\Lambda \xrightarrow{\sim} R$ passes to a quasi-isomorphism 
\[
\Lambda \otimes_{\Lambda} C(E)_{\geq -n} \xrightarrow{\sim} R \otimes_{\Lambda} C(E)_{\geq -n}
\]
and so, increasing $n$, one sees that $R \otimes_{\Lambda} C(E)$ is also acyclic. Its differential takes the form $d + \partial$ with $d$ given by
\begin{equation}\label{differential}
\begin{cases}
    d(\overline{1} \otimes \underline{\eta} \otimes m) & = \sum_{i} \overline{1} \otimes \partial_{\eta_i} \cdot \underline{\eta} \otimes \tau(\eta_i) \cdot m\\
    d(\overline{1} \otimes m) & = \overline{1} \otimes \Omega^{\vee} \otimes \Omega \cdot m \\
    d(\overline{1} \otimes \Omega^{\vee} \otimes \partial_{\underline{\eta}} \otimes m) & = \sum_{i} \overline{1} \otimes \Omega^{\vee} \otimes \partial_{\eta_i} \partial_{\underline{\eta}} \otimes \tau(\eta_i) \cdot m.
\end{cases}
\end{equation}
extended $R$-linearly; hence $R \otimes_{\Lambda} C(E)$ agrees in high degree with $\Gamma_R \otimes_R \overline{E} \xrightarrow{\sim} N$, and so we have constructed a complete resolution of $N^{st}$ over $R$.

\begin{prop}[Buchweitz-Pham-Roberts, \cite{BuchweitzCI}]\label{completeresol} There is a complete resolution of $N^{st}$ over $R$ given by totalising the bicomplex
\[
\xymatrix@C1.6pc@R0.5pc{ \cdots \ar[r] & \Gamma^1_R \otimes_R \overline{E}  \ar[r] &  \overline{E} \ar[r] & \overline{\omega} \otimes_R \overline{E} \ar[r] &  \overline{\omega} \otimes_R T_R^1 \otimes_R \overline{E}  \ar[r] & \cdots
}
\]
\end{prop}

\subsection{Modules of complexity two}\label{complexitytwo}
In this section we describe the complete resolutions and Auslander-Reiten component of indecomposable MCM modules of complexity two.

For each point $p_i \in X$, let $L_i = \Gamma_{\geq 0}(X, \mathcal{O}_{p_i}) = k[X]/I(p_i)$ be the corresponding point module. We may picture $X = V(Q_0) \cap V(Q_\infty)$ as
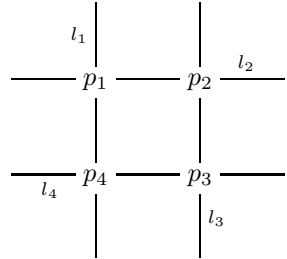
\[
\xymatrix{ 
&  &  & \\
& p_1 \ar@{-}[u]^{l_1} \ar@{-}[l] \ar@{-}[d] \ar@{-}[r] & p_2 \ar@{-}[u] \ar@{-}[d] \ar@{-}[r]^{l_2} & \\
& p_4 \ar@{-}[d] \ar@{-}[l]^{l_4} \ar@{-}[r] & p_3\ar@{-}[d]^{l_3} \ar@{-}[r] & \\
&  &  & \\
}
\]
Here $l_i$ stands for both the line and the corresponding linear form, where 
\begin{align*}
& Q_0 = (x-y)(x+y) = l_1l_3\\
& Q_\infty = (y-z)(y+z) = l_2l_4
\end{align*}
and from the above picture we have $I(p_i) = (l_i, l_{i+1})$, using cyclic indexing. Note that $L_i \cong S/(l_i, l_{i+1})$ has depth $1$ and so is MCM.

Let $\mathbb{K} = \Lambda_S(dQ_0, dQ_\infty) \xrightarrow{\sim} R$ and $E^i = \Lambda_S(dl_i, dl_{i+1}) \xrightarrow{\sim} S/(l_i, l_{i+1}) = L_i$ be the Koszul resolutions, where $|dl_i| = (1,1)$. We have a morphism of dg algebras $\Lambda(dQ_0, dQ_\infty) \to \Lambda(dl_i, dl_{i+1})$ which sends $dQ_j$ to the unique element of the form $l_{j_1}dl_{j_2}$ such that $Q_j = l_{j_1} l_{j_2}$. Similarly, the Koszul complex $E^0 = \Lambda_S(dx, dy, dz) \xrightarrow{\sim} k$ receives a dg morphism from $\Lambda(dQ_0, dQ_\infty)$ by sending $dQ_j \mapsto \gamma_j dx + \epsilon_j dy + \zeta_j dz$, for some choice of linear forms $\gamma_j, \epsilon_j, \zeta_j$ with $Q_j = \gamma_j x + \epsilon_j y + \zeta_j z$.

The morphism of dg algebras $\mathbb{K} \to E^i$ for $0 \leq i \leq 4$ induces a dg module structure on the latter over the former. It follows that the complete resolutions for $k^{st}$ and $L_i$ are given by Prop \ref{completeresol}. Note that we can take $\Omega = dQ_0 dQ_\infty$ for top form, so that $\omega \cong \Lambda (4)[-2]$. We can picture the above resolutions as totalising the bicomplexes 

\[
\xymatrix@C0.5pc@R1.3pc{ \ddots \ar[d] & \ddots \ar[d] & \ddots \ar[d] & \\
 \overline{E}^0_3 \otimes \Gamma^1 \ar[r] & \overline{E}^0_2 \otimes \Gamma^1 \ar[r] \ar[d] & \overline{E}^0_1 \otimes \Gamma^1 \ar[r] \ar[d] & \overline{E}^0_0 \otimes \Gamma^1 \ar[d]\\
& \overline{E}^0_3 \ar[r] & \overline{E}^0_2 \ar[r] & \overline{E}^0_1 \ar[r] \ar[d] & \overline{E}^0_0 \ar[d] \\
&&& \overline{\omega} \otimes \overline{E}^0_3  \ar[r] & \overline{\omega} \otimes \overline{E}^0_2 \ar[r] \ar[d] & 
\overline{\omega} \otimes \overline{E}^0_1  \ar[r] \ar[d] & 
\overline{\omega} \otimes \overline{E}^0_0  \ar[d]\\
&&&&  \overline{\omega} \otimes T^1 \otimes \overline{E}^0_3 \ar[r] &  \overline{\omega} \otimes T^1 \otimes \overline{E}^0_2 \ar[r] \ar[d] & 
 \overline{\omega} \otimes T^1 \otimes \overline{E}^0_1 \ar[r] \ar[d] & 
 \overline{\omega} \otimes T^1 \otimes \overline{E}^0_0 \ar[d]\\
&&&&& \ddots & \ddots & \ddots
}
\]
\[
\xymatrix@C0.5pc@R1.3pc{ \ddots \ar[d] & \ddots \ar[d] & \\
 \overline{E}^i_2 \otimes \Gamma^1 \ar[r]  & \overline{E}^i_1 \otimes \Gamma^1 \ar[r] \ar[d] & \overline{E}^i_0 \otimes \Gamma^1 \ar[d]\\
&  \overline{E}^i_2 \ar[r] & \overline{E}^i_1 \ar[r]  & \overline{E}^i_0 \ar[d] \\
&&&  \overline{\omega} \otimes \overline{E}^i_2 \ar[r]  & 
\overline{\omega} \otimes \overline{E}^i_1  \ar[r] \ar[d] & 
\overline{\omega} \otimes \overline{E}^i_0  \ar[d]\\
&&&&    \overline{\omega} \otimes T^1 \otimes \overline{E}^i_2 \ar[r]  & 
 \overline{\omega} \otimes T^1 \otimes \overline{E}^i_1 \ar[r] \ar[d] & 
 \overline{\omega} \otimes T^1 \otimes \overline{E}^i_0 \ar[d]\\
&&&&& \ddots & \ddots
}
\]
which yield $C(k^{st})$ and $C(L_i)$ respectively.
\begin{lem} The above complete resolutions of $k^{st}$ and $L_i$ are minimal. 
\end{lem}
\begin{proof}
The Koszul resolutions $E^i \xrightarrow{\sim} L_i$ and $E^0 \xrightarrow{\sim} k$ are minimal and the action of $dQ_j: E^i \to E^i(2)[1]$ vanishes after tensoring with $k \otimes_S -$, so it follows from equation \ref{differential} that $\left(R \otimes_{\Lambda} C(E^i), d \right)$ is a minimal complete resolution for $i = 0,1,2,3,4$.
\end{proof}


The MCM modules $k^{st}$ and $L_i$ account for essentially all modules of complexity $2$. The following will be proven in section \ref{proofsection}.
\begin{thm}\label{pointmodules}
Let $M$ be an indecomposable graded MCM module of complexity $2$.
Then $M$ is one of the following:
\begin{enumerate}[1.]
\item $syz_R^n \big(k^{st}\big)(m)$ 
\item $syz_R^n \big(L_i\big)(m)$
\end{enumerate}
for $i = 1,2,3,4$ and $m,n \in \Z$.
\end{thm} 

\begin{thm} The Auslander-Reiten component $\mathcal{C}$ in $\underline{{\rm MCM}}(gr R)$ containing $k^{st}, L_i$ has the form
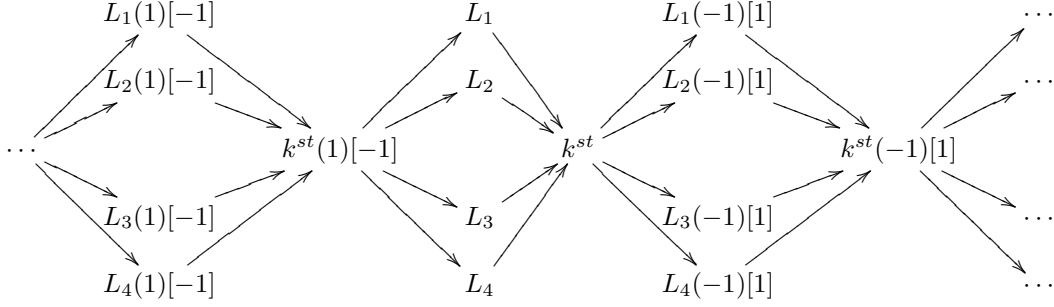
\[
\xymatrix@R.8pc@C1.6pc{
&  L_1(1)[-1] \ar[ddr] && L_1 \ar[ddr] && L_1(-1)[1] \ar[ddr] && \cdots\\
& L_2(1)[-1] \ar[dr] && L_2 \ar[dr] && L_2(-1)[1]  \ar[dr] && \cdots\\
\cdots \ar[uur] \ar[ur] \ar[dr] \ar[ddr] && k^{st}(1)[-1] \ar[uur] \ar[ur] \ar[dr] \ar[ddr] && k^{st}\ar[uur] \ar[ur] \ar[dr] \ar[ddr] && k^{st}(-1)[1]\ar[uur] \ar[ur] \ar[dr] \ar[ddr] &  \\
& L_3(1)[-1] \ar[ur] && L_3 \ar[ur] && L_3(-1)[1] \ar[ur] && \cdots\\
& L_4(1)[-1] \ar[uur] && L_4 \ar[uur] && L_4(-1)[1] \ar[uur] && \cdots \\
}
\]

and all indecomposables of complexity two fall in one of $\mathcal{C}[n]$ for some $n \in \Z$.
\end{thm} 

\subsection{Modules of complexity one}\label{complexityone}

Assuming the equivalence of categories ${\rm E}: \underline{{\rm MCM}}(gr R)^{op} \xrightarrow{\cong} {\rm D}^b(kQ)$, we may ponder the construction of MCM modules corresponding to regular representations of $kQ$. Such modules are characterized by being $\tau$-periodic, and we have seen that they vary in families.

Since ${\rm E} \circ \tau^{-1} = \tau \circ {\rm E}$ and $\tau = (1)[-1]$ in $\underline{{\rm MCM}}(gr R)$, regular representations correspond to $\tau$-periodic MCM modules, or equivalently modules with periodic minimal resolutions and thus of complexity one. Eisenbud characterized such modules over complete intersections. Recall that a matrix factorization of a regular element $f$ in a graded Gorenstein commutative ring $B$ is a pair of square matrices $(\varphi, \psi)$ with homogeneous entries in $B$ such that $\varphi \psi = f \cdot {\rm I}_n = \psi \varphi$, and which yield homogeneous maps
\[
\bigoplus_{i=1}^{n} B(n_i) \xrightarrow{\varphi} \bigoplus_{i=1}^{n} B(m_i) \xrightarrow{\psi} \bigoplus_{i=1}^{n} B(n_i + |f|)
\]
The matrices $(\varphi, \psi)$ descend to differentials on a $2$-periodic complex of graded free modules $\overline{B} = B/(f)$, which is acyclic by regularity of $f$, and so yields a complete resolution of the then MCM module ${\rm coker}(\overline{\varphi})$ over $\overline{B}$.

\begin{prop}[Eisenbud, \cite{MR570778} Thm 5.2]
Let $R = S/I$ be a complete intersection with $S$ regular local. Let $M$ be an MCM $R$-module with minimal free resolution ${\rm F}_*$ of bounded rank. Then there is a regular sequence $(g_1, ..., g_c)$ generating $I$ such that ${\rm F}_*$ is given by a matrix factorization of $g_c$ over $S/(g_1, ..., g_{c-1})$.
\end{prop}
A version of this theorem holds for graded rings and modules, and this inspires the following construction. Recall that the pencil of quadrics through $X$ has singular quadrics given by
\begin{align*}
    & Q_0 = x^2 - y^2\\
    & Q_1 = x^2 - z^2\\
    & Q_{\infty} = y^2 - z^2
\end{align*}
with general quadric given by $Q_{t} = t_0Q_{\infty} + t_1 Q_0$ for $t = (t_0, t_1) \in k^2\setminus \{0 \}$. Let $(\Phi_{t}^{+}, \Phi_{t}^{-})$ be the pair of matrices over $S$ given by
\begin{align*}
\Phi_{t}^{+} = & \begin{bmatrix} t_1 (x+y) & y+z \\
										 t_0(z-y) & x-y \end{bmatrix}\\
										 \\
\Phi_{t}^{-} = & \begin{bmatrix} x-y & -y-z \\
                                        t_0(y-z)  & t_1(x+y)  \end{bmatrix}
\end{align*}
We have $\Phi_{t}^{+} \Phi_{t}^{-} = Q_{t} \cdot {\rm I}_2 = \Phi_{t}^{-} \Phi_{t}^{+}$. Letting $s = (s_0, s_1)$ correspond to a different point in $\mathbb{P}^1$, the sequence $(Q_{s}, Q_{t})$ is regular and the pair $(\Phi^{+}_{t}, \Phi^{-}_{t})$ defines a matrix factorization of $Q_{t}$ over $S' = S/(Q_{s})$. The module
\[
N_{t} = {\rm coker}\big(R(-1)^{\oplus 2} \xrightarrow{\overline{\Phi}_{t}^{+}} R^{\oplus 2} \big)
\]
is then an MCM module over $R \cong S'/(Q_t)$, with isomorphism class independent of representative of $[t_0: t_1] \in \mathbb{P}^1$. We call the family $\mathcal{Q} = \{N_{t} \}_{t \in \mathbb{P}^1}$ the {\bf fundamental family}.

\begin{prop}\label{Nindec} We have $\underline{{\rm End}}_{gr R}(N_{t}) = k$ for every $t \in \mathbb{P}^1$, thus the modules $N_{t}$ are indecomposable.
\end{prop}
\begin{proof}
Direct calculations show that the only scalar matrices $A, B$ fitting in a commutative diagram
\[
\xymatrix@R1pc@C2pc{
R(-1)^{\oplus 2} \ar[d]^A \ar[r]^{\Phi_{t}^{+}} & R^{\oplus 2} \ar[d]^B \\
R(-1)^{\oplus 2} \ar[r]^{\Phi_{t}^{+}} & R^{\oplus 2}
}
\]
are given by $A = B = c \cdot{\rm I}_2$ for some $c \in k$, for any $t \in \mathbb{P}^1$. Thus ${\rm End}_{gr R}(N_{t}) = \underline{{\rm End}}_{gr R}(N_{t}) = k$.
\end{proof}

Next, consider $t = 0, 1, \infty$ corresponding to the $3$ singular quadrics listed above. Then the factorizations 
\begin{align*}
    & Q_0 = x^2 - y^2 = (x-y)(x+y)\\
    & Q_1 = x^2 - z^2 = (x-z)(x+z)\\
    & Q_{\infty} = y^2 - z^2 = (y-z)(y+z)
\end{align*}
are size one matrix factorizations of $Q_{t}$, netting us additional MCM modules\footnote{The seemingly odd naming convention will become clear later.}
\begin{align*}\label{Dpm}
    D_0^{+} = {\rm coker}\big( R(-1) \xrightarrow{x-y} R\big) &&  D_0^{-} = {\rm coker}\big( R(-1) \xrightarrow{x+y} R\big)\\
    D_1^{+} = {\rm coker}\big( R(-1) \xrightarrow{x+z} R\big) &&  D_1^{-} = {\rm coker}\big( R(-1) \xrightarrow{x-z} R\big)\\
    D_{\infty}^{+} = {\rm coker}\big( R(-1) \xrightarrow{y+z} R\big) &&  D_{\infty}^{-} = {\rm coker}\big( R(-1) \xrightarrow{y-z} R\big)
\end{align*} 

\begin{prop}\label{D_t} These have the following properties:
\begin{enumerate}[i.]
\item We have $\underline{{\rm End}}_{gr R}(D_{t}^{\pm}) = k$, hence the modules $D_{t}^{\pm}$ are indecomposable.
\item We have $syz^1_R \big(D_{t}^{+} \big) \cong D_{t}^{-}(-1)$ and $syz^1_R \big( D_{t}^{-} \big) \cong D_{t}^{+}(-1)$.
\end{enumerate}
\end{prop}
\begin{proof}
We have ${\rm End}_{gr R}(D_{t}^{\pm}) \cong \big(R/l\big)_0 = k$ with $l$ the corresponding linear form, hence $i.$ follows, and $ii.$ is by construction.
\end{proof}

We will show in section $\ref{proofsection}$, Thm \ref{maincalc} that $N_{t}$ correspond to $R_{t}$ and $D_{t}^{\pm}$ to $S_{t}^{\pm}$ under the equivalence ${\rm E}$ of Theorem \ref{E}. The next four results will then follow immediately from Prop \ref{R_t} through Theorem \ref{ARquiverQ}. 

\begin{prop}\label{N_t} Let $t \in \mathbb{P}^1 \setminus \{0, 1, \infty \}$. Then:
\begin{enumerate}[i.]
\item There are unique non-trivial extensions
\[
N_{t} \to N_{t} \langle 2 \rangle \to N_{t} \to N_{t}[1]
\]
\[
N_{t} \to N_{t} \langle r+1 \rangle \to N_{t} \langle r \rangle \to N_{t}[1]
\]

and $N_{t}\langle r \rangle$ are indecomposable for all $r \geq 2$.
\item Letting $N_{t}\langle 1 \rangle := N_{t}$,  we have $syz^1_R \big(N_{t} \langle r \rangle \big) \cong N_{t} \langle r \rangle (-1)$ for all $r \geq 1$. Hence $N_{t}\langle r \rangle$ have $1$-periodic linear minimal resolutions.
\end{enumerate}
\end{prop}

Next, let us denote $N_{t}^{+} := N_{t} = {\rm coker}\big(\Phi_{t}^{+} \big)$ and $N_{t}^{-} := {\rm coker} \big(\Phi^{-}_{t} \big)$\footnote{Implicitly generated in degree zero.}. For $t \in \mathbb{P}^1 \setminus \{0, 1, \infty \}$ the previous proposition says that $N_{t}^{+} \cong N_{t}^{-}$. We now consider the other cases.

\begin{prop} Let $t = 0, 1, \infty$. Then:
\begin{enumerate}[i.]
\item $N_{t}^+ \ncong N_{t}^-$.
\item There are (unique) non-trivial extensions
\[
D_{t}^{-} \to N_{t}^{+} \to D_{t}^{+} \to D_{t}^{-}[1]
\]
\[
D_{t}^{+} \to N_{t}^{-} \to D_{t}^{-} \to D_{t}^{+}[1]
\]
\item Letting $D_{t}\langle 2 \rangle^{\pm}$ stand for $N_{t}^{\pm}$, there are unique non-trivial extensions
\[
D_{t}^{+} \to D_{t}\langle r+1 \rangle^{-} \to D_{t} \langle r \rangle^{-} \to D_{t}^{+}[1]
\]
\[
D_{t}^{-} \to D_{t}\langle r+1 \rangle^{+} \to D_{t} \langle r \rangle^{+} \to D_{t}^{-}[1]
\]
for all $r \geq 2$ with $D_{t}\langle r \rangle^{\pm}$ indecomposable.
\item We have $syz^1_R(D_t\langle r \rangle)^{\pm}) \cong D_t\langle r \rangle^{\mp}(-1)$, hence $D_t\langle r \rangle^{\pm}$ have $2$-periodic linear minimal resolutions.
\end{enumerate}
\end{prop}

We can see the extensions of $ii.$ as follows. Define matrix factorizations $(\Psi_t^+, \Psi_t^-)$ for $t = 0,1, \infty$ by
\begin{align*}
\Psi_{0}^{+} = \begin{bmatrix} x+y & z \\
										 0 & x-y \end{bmatrix}
										 &&
\Psi_{0}^{-} = \begin{bmatrix} x-y &  -z\\
                                        0  & x+y  \end{bmatrix}
                                        \\
                                        \\
\Psi_{1}^{+} = \begin{bmatrix} x-z & y \\
										 0 & x+z  \end{bmatrix}
										 &&
\Psi_{1}^{-} = \begin{bmatrix} x+z & -y \\
                                        0 & x-z \end{bmatrix}
                                        \\
                                        \\
\Psi_{\infty}^{+} = \begin{bmatrix} y-z & x \\
                                        0  & y+z \end{bmatrix}
										 &&
\Psi_{\infty}^{-} = \begin{bmatrix} y+z & -x \\
                                        0 & y-z \end{bmatrix}
\end{align*}
One easily sees that ${\rm coker}(\Psi_t^{\pm}) \cong {\rm coker}(\Phi_t^{\pm}) = N_t^{\pm}$, and these give rise to said non-trivial extensions. It would be interesting to have an explicit presentations of the higher extensions.

Finally, let $\mathcal{Q}_{t}$ denote the Auslander-Reiten component containing $N_{t}$ in $\underline{{\rm MCM}}(gr R)$.

\begin{prop} The components $\mathcal{Q}_{t}$ are given by the quivers below with the edges identified:

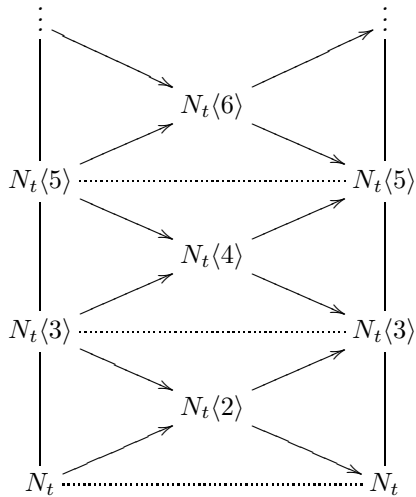
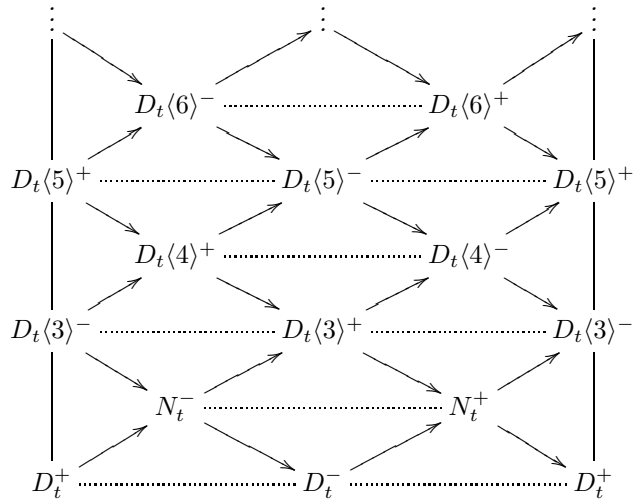
\begin{align*}
\xymatrix@R1pc@C.8pc{\vdots \ar[dr] \ar@{-}[dd] && \vdots \ar@{-}[dd] &&  \vdots \ar[dr] \ar@{-}[dd] && \vdots \ar[dr] && \vdots \ar@{-}[dd] \\
 & N_{t} \langle 6 \rangle \ar[ur] \ar[dr] & && & D_{t}\langle 6 \rangle^{-} \ar[dr] \ar[ur] \ar@{.}[rr] && D_{t}\langle 6 \rangle^{+} \ar[dr] \ar[ur] & \\
 N_{t}\langle 5 \rangle \ar[dr] \ar[ur] \ar@{.}[rr] \ar@{-}[dd] && N_{t} \langle 5 \rangle \ar@{-}[dd] && D_{t}\langle 5 \rangle^{+} \ar[dr] \ar[ur] \ar@{-}[dd] \ar@{.}[rr] && D_{t}\langle 5 \rangle^{-} \ar[dr] \ar[ur] \ar@{.}[rr] && D_{t}\langle 5 \rangle^{+} \ar@{-}[dd] \\
 & N_{t}\langle 4 \rangle \ar[dr] \ar[ur] & && & D_{t}\langle 4 \rangle^{+} \ar[dr] \ar@{.}[rr] \ar[ur] && D_{t}\langle 4 \rangle^{-} \ar[dr] \ar[ur] \\
 N_{t}\langle 3 \rangle \ar[dr] \ar[ur] \ar@{.}[rr] \ar@{-}[dd] && N_{t}\langle 3 \rangle \ar@{-}[dd] && D_{t}\langle 3 \rangle^{-} \ar@{-}[dd] \ar@{.}[rr] \ar[dr] \ar[ur] && D_{t}\langle 3 \rangle^{+} \ar[dr] \ar[ur] \ar@{.}[rr] && D_{t}\langle 3 \rangle^{-} \ar@{-}[dd] \\
 & N_{t} \langle 2 \rangle \ar[dr] \ar[ur] & && & N_{t}^{-} \ar[dr] \ar[ur] \ar@{.}[rr] && N_{t}^{+} \ar[dr] \ar[ur] \\
 N_{t} \ar[ur] \ar@{.}[rr] && N_{t} && D_{t}^{+} \ar[ur] \ar@{.}[rr] && D_{t}^{-} \ar[ur] \ar@{.}[rr] && D_{t}^{+}\\
 & t \in \mathbb{P}^1 \setminus \{0, 1, \infty \} & && && t = 0, 1, \infty. &&
 }
\end{align*}

\end{prop}

\begin{thm}\label{ARquiverMCM}
The Auslander-Reiten quiver of $\underline{{\rm MCM}}(gr R)$ is given by the union of disjoint components
\[
\Gamma\big(\underline{{\rm MCM}}(gr R)\big) = \bigg(\bigcup_{n \in \Z} \mathcal{C}[n] \bigg) \cup \bigg(\bigcup_{t \in \mathbb{P}^1, n \in \Z} \mathcal{Q}_{t}[n]\bigg).
\]

Hence the classification of indecomposable graded MCM modules is as listed in Theorem \ref{maintheorem}.
\end{thm}

Now let $\widehat{R}$ be the completion at the irrelevant ideal. According to \cite[Prop 1.5]{MR2776613}, the completion functor $\underline{{\rm MCM}}(gr R) \to \underline{{\rm MCM}}(\widehat{R})$ can be identified with canonical morphism to the triangulated hull of the orbit category of $\underline{{\rm MCM}}(gr R)$ under $F = (1)$. Since $\underline{{\rm MCM}}(gr R) \cong {\rm D}^b(kQ)$ and $(1) = \tau[1]$, it follows from results of Keller \cite{MR2184464} that the completion functor is essentially surjective, and so every MCM module over $\widehat{R}$ is gradable. Our result then extends to a description of indecomposables in $\underline{{\rm MCM}}(\widehat{R})$.

\subsection{Betti tables and Ulrich modules}
Let $M$ be a graded $R$-module with minimal complete free resolution
\[
\cdots \to \bigoplus_{j \in \Z} R(-j)^{\oplus \beta_{i, j}} \to \cdots \to \bigoplus_{j \in \Z} R(-j)^{\oplus \beta_{1, j}} \to \bigoplus_{j \in \Z} R(-j)^{\oplus \beta_{0, j}} \to \bigoplus_{j \in \Z} R(-j)^{\oplus \beta_{-1, j}} \to \cdots.
\]
The Betti table of $M$ is the table whose $i^{th}$ column and $j^{th}$ row is given by $\beta_{i, i+j}$. We list here all Betti tables of indecomposables, which can be calculated directly from the previous two sections. Note that $syz^{\pm 1}_R(-)(\pm1)$ and $(\pm1)$ correspond to horizontal and vertical shifts respectively.

\begin{prop}\label{betti} The Betti tables of indecomposable graded MCM modules are given up to syzygy and degree shifts by:
\begin{align*}
N_{t}\langle r \rangle: && D_{t}\langle r \rangle^{\pm}: \\
& \begin{tabular}{ c | c  c  c  c  c  c }
 & \ldots & 2 & 1 & 0 & -1 & \ldots \\
\hline
 -1 & - & - & - & - & - & - \\
  0 & \ldots & 2r & 2r & 2r & 2r & \ldots \\
  1 & - & - & - & -  & - & -\\
\end{tabular}
&& \begin{tabular}{ c | c  c  c  c  c  c }
 & \ldots & 2 & 1 & 0 & -1 & \ldots \\
\hline
 -1 & - & - & - & - & - & - \\
  0 & \ldots & r & r & r & r & \ldots \\
  1 & - & - & - & -  & - & -\\
\end{tabular}
\end{align*}

\begin{align*}
 k^{st}: &&  L_i: & \\
 & \begin{tabular}{ c | c  c  c  c  c  c  c  c }
 & \ldots & 3 & 2 & 1 & 0 & -1 & -2 & \ldots \\
\hline
 --2 & - & - & - & - & - & - & - & -\\
 -1 & - & - & - & - & 1 & 3 & 5 & \ldots \\
  0 & \ldots & 7 & 5 & 3 & 1 & - & - & - \\
  1 & - & - & - & - & -  & - & - & - \\
\end{tabular}
&& \begin{tabular}{ c | c  c  c  c  c  c  c  c }
 & \ldots & 3 & 2 & 1 & 0 & -1 & -2 & \ldots \\
\hline
 -2 & - & - & - & - & - & - & - & -\\
 -1 & - & - & - & - & - & 1 & 2 & \ldots \\
  0 & \ldots & 4 & 3 & 2 & 1 & - & - & - \\
  1 & - & - & - & - & -  & - & - & - \\
\end{tabular}
\end{align*}
\end{prop}
Note that from this list, one sees that all finitely generated $R$-modules have minimal resolutions which are eventually linear. This result holds in general for Koszul rings, as shown by Avramov and Eisenbud in \cite{MR1195407}.

Next we record the main numerical invariants of non-free graded MCM modules. For such a module $M$, the Hilbert series takes the form
\[
H_M(t) = \frac{Q_M(t)}{1-t}
\]
where $Q_M(t)$ is a Laurent polynomial with $Q_M(1) \neq 0$. We let $\nu(M)$ be the minimal number of generators of $M$ and $e(M) = Q_M(1)$ the multiplicity. These are invariant under degree shift, and we can work these out from the betti tables.
\begin{prop} The numerical invariants of indecomposable graded MCM modules are given by
\[
\centerline{
\begin{tabular}{|c|c|c|c|}
  \hline
  $M$ & $Q_M(t)$ & $\nu(M)$ & $e(M)$ \\
  \hline
  &&& \\
  $N_t\langle r \rangle$ & 2r(1+t) & 2r & 4r\\
  $D_t\langle r \rangle^{\pm}$ & r(1+t) & r & 2r\\
  &&& \\
  $syz^n(k^{st}), n \geq 1$ & $(2n-1)t^{n+1} + (2n+1)t^{n}$ &2n+1 & 4n \\
  $k^{st}$ & $3 + t^{-1}$ & 2 & 4 \\
  $syz^{-n}(k^{st}), n \geq 1$ & $(2n+3)t^{-n} + (2n+1)t^{-(n+1)}$ & 2n+1 & 4n+4 \\
  &&& \\
  $syz^n(L_i), n \geq 0$ & $nt^{n+1} + (n+1)t^n$ & n+1 & 2n+1 \\
  $syz^{-n}(L_i), n \geq 1$ & $(-1)^{n+1}(t-1) + (n+1)t^{-(n-1)} + nt^{-n}$ & n & 2n+1\\
  &&& \\
  \hline
\end{tabular}}
\]
\end{prop}

There is a general upper bound $\nu(M) \leq e(M)$. Ulrich modules are the modules which meet this bound.
\begin{cor} Up to degree shifts, the indecomposable graded Ulrich modules are given by the point modules $L_i$.
\end{cor}

\section{Proof of the classification result}\label{proofsection}
\subsection{The $\widetilde{D}_4$ tilting module}

In this section, we prove the equivalence of categories stated in the introduction. We make use of standard results coming from tilting theory.
\begin{defn}
Let $\mathcal{T}$ be a triangulated category. We say that $T \in \mathcal{T}$ is tilting if ${\rm thick}(T) = \mathcal{T}$ and
$T$ has no self-extensions, i.e. ${\rm Hom}_{\mathcal{T}}(T, T[n]) = 0$ for $n \neq 0$.
\end{defn} 

\begin{thm}[Keller, \cite{MR2384608}, Thm 8.7]
Let $\mathcal{T}$ be an idempotent closed, algebraic triangulated category with dg enhancement $\mathcal{C}$, and with tilting object $T$. Let $\mathcal{A} = {\rm Hom}_{\mathcal{C}}(T, T)$, and $A = {\rm H}^0(\mathcal{A}) = {\rm Hom}_{\mathcal{T}}(T, T)$. Then there are equivalences of triangulated categories
\begin{enumerate}[i.]
\item ${\rm Hom}_{\mathcal{C}}(-, T): \mathcal{T}^{op} \xrightarrow{\cong} {\rm D}_{{\rm perf}}(\mathcal{A})$,
\item $\Theta: {\rm D}_{{\rm perf}}(\mathcal{A}) \xrightarrow{\cong} {\rm D}_{{\rm perf}}(A)$.
\end{enumerate}
\end{thm}
The equivalence $\Theta$ follows from base change over the quasi-isomorphisms $\mathcal{A} \xleftarrow{\sim} \tau^{\leq 0}\mathcal{A} \xrightarrow{\sim} {\rm H}^0(\mathcal{A}) = A$. The first equivalence is usually stated covariantly by using ${\rm Hom}_{\mathcal{C}}(T, -)$, however we shall find the contravariant version better suited for calculations.

Letting $\mathcal{T} = \underline{{\rm MCM}}(gr R)$, note that $\mathcal{T}$ is Krull-Schmidt and thus idempotent closed, and algebraic as we may take $\mathcal{C}$ to be the dg category of complete resolutions, so that ${\rm H}^0(\mathcal{C}) = \mathcal{K}^{\infty}_{ac}(proj R) \cong \underline{{\rm MCM}}(gr R)$. We will write ${\rm R\underline{Hom}}(M, N) = {\rm Hom}_{\mathcal{C}}(C(M), C(N))$, where $C(M)$ and $C(N)$ are complete resolutions of $M, N$. We are left with providing the tilting module $T$.

Most tilting results for MCM modules go through Orlov's Theorem, and this paper is no exception; we briefly recall the result in the form we need. Let $B = k[Z]$ be the homogeneous coordinate ring of a projective variety $Z$. Let ${\rm D}^{b}(gr_{\ge i} B) \subset {\rm D}^{b}(gr B)$ be the subcategory consisting of complexes whose cohomology modules ${\rm H}^n$ are generated in degree $\geq i$. Then ${\rm R\Gamma}_{\geq i}(Z, -): {\rm D}^b(Z) \to {\rm D}^b(gr_{\ge i} B)$ can be shown to be fully faithful.
\begin{thm}[Orlov, \cite{MR2641200}] Assume that $B$ is Gorenstein with $a > 0$. Then for each $i \in \Z$ there is an embedding of triangulated categories
\[
\Psi_i: \underline{{\rm MCM}}(gr B) \to {\rm D}^b(gr_{\ge i} B)
\]
right inverse to MCM approximation, and a semiorthogonal decomposition
\[
\Psi_i\big(\underline{{\rm MCM}}(gr B)\big) = \big\langle k(-i), k(-i-1), ..., k(-i-a+1), {\rm R\Gamma}_{\geq i+a}({\rm D}^b(Z)) \big\rangle.
\]
\end{thm}

We will need a description of $\Psi_i$ for calculations.
\begin{lem}[Buchweitz, Burke-Stevenson {\cite[Sect. 5]{MR3496862}}]\label{unstabilization} Let $M$ be a graded MCM $B$-module with complete resolution $C$. Let $C_{[\geq i]}$ be the complex of graded free modules obtained from $C$ by removing free module summands generated in degrees $< i$. Then $\Psi_i(M) \cong C_{[\geq i]}$ in ${\rm D}^b(gr_{\geq i} B)$. 
\end{lem} 

We can now describe the tilting object $T$.
\begin{thm}\label{tiltingobject}
The MCM module $T = \left(\bigoplus_{i = 1}^{4} L_i \right) \oplus k^{st}$ is a tilting module for $\underline{{\rm MCM}}(gr R)$ with endomorphism algebra $A = kQ$, with $Q$ as in Section $2$.
\end{thm}
\begin{proof} 
Applying Orlov's Theorem to $Z = X = \{p_1, p_2, p_3, p_4 \}$ with $i = -a = -1$, we have a tilting object $\left(\bigoplus_{i = 1}^4 \mathcal{O}_{p_i}\right)$ for ${\rm D}^b(X)$, and ${\rm R\Gamma}_{\ge i+a}(X, \mathcal{O}_{p_i}) = \Gamma_{\ge 0}(X, \mathcal{O}_{p_i}) = k[X]/I(p_i) = L_i$. This gives
\[
\underline{{\rm MCM}}(gr R) = \big \langle k^{st}(1), L_1, L_2, L_3, L_4 \big \rangle.
\]
This is already a full strong exceptional collection, but for convenience we mutate the collection to move $k^{st}(1)$ to the rightmost position. By Helix Theory \cite{MR1074777}, the resulting collection is (ignoring suspensions) 
\[
\big \langle L_1, L_2, L_3, L_4, k^{st}(1) \otimes_R \omega_R^{-1} \big \rangle = \big \langle L_1, L_2, L_3, L_4, k^{st} \big \rangle
\] 
which thus remains a full exceptional sequence. We may apply $\Psi_0$ to perform calculations in ${\rm D}^b(gr_{\ge 0} R)$, and by the previous lemma and \ref{betti} we find that $\Psi_0(k^{st}) = k$ and $\Psi_0(L_i) = L_i$. It follows that $U = \left(\bigoplus_{i = 1}^4 L_i \right) \oplus k$ generates $\Psi_0\big(\underline{{\rm MCM}}(gr R)\big)$. Only the extension groups ${\rm Ext}_{gr R}^n(L_i, k)$ are relevant, with dimensions given by
\[
{\rm dim}\ {\rm Ext}_{gr R}^n(L_i, k) = {\rm dim}\ {\rm Tor}^R_n(L_i, k)_0 = \begin{cases} 1, & n = 0\\
0, & n \neq 0.
\end{cases}
\]
Then $T = U^{st}$ is a tilting module since $(-)^{st}$ is an inverse for $\Psi_0$, and $\underline{{\rm End}}_{gr R}(T) \cong {\rm End}_{gr R}(U) \cong kQ$.
\end{proof}
\begin{cor}
There is an equivalence of triangulated categories
\[
{\rm E} = \Theta \circ {\rm R\underline{Hom}}(-, T): \underline{{\rm MCM}}(gr R)^{op} \xrightarrow{\cong} {\rm D}^b(kQ).
\]

\end{cor}


\subsection{Reduction to the Four Subspace problem}
We are left with calculating the image of our proposed list of indecomposable modules under ${\rm E}$. 

\begin{lem}\label{tiltingcalc} Let $M$ be a MCM module such that ${\rm H}^i\left({\rm R\underline{Hom}}(M, T)\right) = 0$ for $i \neq 0$. Then 
\[
{\rm E}(M) = \Theta \circ {\rm R\underline{Hom}}(M, T) \cong \underline{{\rm Hom}}_{gr R}(M, T)
\]
in ${\rm D}^b(A)$, where $A = \underline{{\rm End}}_{gr R}(T)$ acts naturally on the left of the latter.
\end{lem}
\begin{proof}
By assumption ${\rm H}^i \left(\Theta \circ {\rm R\underline{Hom}}(M, T)\right) = 0$ for $i \neq 0$ in ${\rm D}^b(A)$ so ${\rm E}(M)$ must be quasi-isomorphic to its zero cohomology $\underline{{\rm Hom}}_{gr R}(M, T)$, and it isn't hard to see that the endomorphism action is compatible through $\Theta$.
\end{proof}

Let us prove the results of section $3.1$. First, note that indecomposable summands of $T$ are sent to indecomposable summands of $kQ$, so that ${\rm E}(k^{st}) = P(0)$ and ${\rm E}(L_i) = P(i)$ for $i = 1, 2, 3, 4$. The complexes in the preprojective-preinjective components are all of the form $\tau^{m} P(i)[n]$, which corresponds under ${\rm E}$ to $syz_R^{m-n} k^{st}(-m)$ and $syz_R^{m-n} L_i(-m)$. Keeping in mind contravariance of ${\rm E}$, the component $\mathcal{C}$ is sent onto $\mathcal{PI}$. We have already seen that regular representations of $Q$ must correspond to MCM modules of complexity one, thus we have classified indecomposables of complexity two.

We now prove the results of section $3.2$. This will follow from the next theorem, which constitutes the main result of this paper.
\begin{thm}\label{maincalc} We have the following isomorphisms in ${\rm D}^b(A)$:
\begin{enumerate}[i.]
\item ${\rm E}(N_{t}) \cong R_{t}$ for all $t \in \mathbb{P}^1$.
\item ${\rm E}(D_{t}^{\pm}) \cong S_{t}^{\pm}$ for $t = 0, 1, \infty$.
\end{enumerate}
\end{thm}
For ease of calculations note that we have $\Psi_0(L_i) = L_i$, $\Psi_0(k^{st}) = k$, $\Psi_0(N_{t}) = N_{t}$ and $\Psi_0(D_{t}^{\pm}) = D_{t}^{\pm}$, which follows from Lemma \ref{unstabilization}.

\begin{proof}
We first show $i$. Let $U = \Psi_0(T) = \big(\bigoplus_{i = 1}^4 L_i \big) \oplus k$. We have
\[
\underline{{\rm Hom}}_{gr R}(N_{t}, T) \cong {\rm Hom}_{gr R}(N_{t}, U) \twoheadrightarrow {\rm Hom}_{gr R}(N_{t}, k) = {\rm D}{\rm Tor}_0^{R}(N_{t}, k)_0 \neq 0
\]
where ${\rm D}$ denotes vector space dual. This means that ${\rm H}^0({\rm R\underline{Hom}}(N_{t}, T)) \neq 0$ and so ${\rm H}^i({\rm R\underline{Hom}}(N_{t}, T)) = 0$ for $i \neq 0$ since $N_{t}$ is indecomposable. By Lemma \ref{tiltingcalc} we have to compute the module structure on $\underline{{\rm Hom}}_{gr R}(N_{t}, T) \cong {\rm Hom}_{gr R}(N_{t}, U)$ under endomorphisms of $U$, or equivalently the maps in the diagram
\[
\xymatrix@R1.2pc@C.5pc{
{\rm Hom}_{gr R}(N_{t}, L_1) \ar[drr] & {\rm Hom}_{gr R}(N_{t}, L_2) \ar[dr] && {\rm Hom}_{gr R}(N_{t}, L_3) \ar[dl] & {\rm Hom}_{gr R}(N_{t}, L_4) \ar[dll] \\
& & {\rm Hom}_{gr R}(N_{t}, k)
}
\]
induced by the canonical quotient $L_i \twoheadrightarrow k$. 

Present $L_i = k[X]/I(p_i) = R/(l_i, l_{i+1})$ as a quotient of two linear forms as in section \ref{complexitytwo}. We will show that ${\rm dim}_k {\rm Hom}_{gr R}(N_{t}, L_i) = 1$ with generators given by
\[
\xymatrix@C2.8pc{ R(-1)^{\oplus 2} \ar[r]^{\Phi_{t}^{+}} \ar[d]^{\begin{bsmallmatrix} 0 & 1\\ -t_0 & 0 \end{bsmallmatrix}} & R^{\oplus 2} \ar[d]^{\begin{bsmallmatrix} 0 & 1 \end{bsmallmatrix}} \ar@{->>}[r] & N_{t} \ar@{.>}[d] && R(-1)^{\oplus 2} \ar[r]^{\Phi_{t}^{+}} \ar[d]^{\begin{bsmallmatrix} t_1 & 1\\ -t_0 & -1 \end{bsmallmatrix}} & R^{\oplus 2} \ar[d]^{\begin{bsmallmatrix} 1 & 1 \end{bsmallmatrix}} \ar@{->>}[r] & N_{t} \ar@{.>}[d]\\
R(-1)^{\oplus 2} \ar[r]_{\begin{bsmallmatrix} x-y & y-z \end{bsmallmatrix}} & R \ar@{->>}[r] & L_1 && R(-1)^{\oplus 2} \ar[r]_{\begin{bsmallmatrix} x+y & y-z \end{bsmallmatrix}} & R \ar@{->>}[r] & L_2 \\
}
\]

\[
\xymatrix@C2.8pc{ R(-1)^{\oplus 2} \ar[r]^{\Phi_{t}^{+}} \ar[d]^{\begin{bsmallmatrix} t_1 & 1\\ 0 & 0 \end{bsmallmatrix}} & R^{\oplus 2} \ar[d]^{\begin{bsmallmatrix} 1 & 0 \end{bsmallmatrix}} \ar@{->>}[r] & N_{t} \ar@{.>}[d] && R(-1)^{\oplus 2} \ar[r]^{\Phi_{t}^{+}} \ar[d]^{\begin{bsmallmatrix} t_0t_1 & t_1 \\ t_0t_1 & t_0 \end{bsmallmatrix}} & R^{\oplus 2} \ar[d]^{\begin{bsmallmatrix} t_0 & t_1 \end{bsmallmatrix}} \ar@{->>}[r] & N_{t} \ar@{.>}[d]\\
R(-1)^{\oplus 2} \ar[r]_{\begin{bsmallmatrix} x+y & y+z \end{bsmallmatrix}} & R \ar@{->>}[r] & L_3 && R(-1)^{\oplus 2} \ar[r]_{\begin{bsmallmatrix} x-y & y+z \end{bsmallmatrix}} & R \ar@{->>}[r] & L_4 \\
}
\]
where we recall that $\Phi_{t}^{+} = \begin{bsmallmatrix} t_1 (x+y) & y+z \\
										 t_0(z-y) & x-y
										 \end{bsmallmatrix}$. 
The above chain-maps are homotopically non-trivial for degree reasons, hence \[
d_i := {\rm dim}_k {\rm Hom}_{gr R}(N_t, L_i) \geq 1.
\]
Note that $d_0 := {\rm dim}_k {\rm Hom}_{gr R}(N_{t}, k) = 2$ with the obvious basis. Now, one could verify by hand that the above morphisms are the only possible ones. Alternatively, note that $N_t$ must be sent to a regular module since it has complexity one. Let
\[
\partial_A: K_0(A) \to \Z
\]
be the defect of the algebra $A = kQ$, which takes a module $M$ with dimension vector $(d_0, d_1, d_2, d_3, d_4)$ to 
\[
-2d_0 + \sum_{i=1}^4 d_i.
\]
By \cite[Chp XI.2, Prop 2.3]{MR2360503} and \cite[Chp XIII.3 Lemma 3.4]{MR2360503}, for $M$ indecomposable one has $\partial_A(M) = 0$ if and only if $M$ is regular. It follows that $\partial_A({\rm E}(N_t)) = 0$ and so $d_1 = \cdots = d_4 = 1$.

Finally since ${\rm dim}_k {\rm Hom}_{gr R}(N_t, L_i) = 1$ with the above maps as basis, composing these with the canonical quotient $L_i \twoheadrightarrow k$ shows that ${\rm E}(N_{t}) = R_{t}$.

Establishing $ii.$ is easier. As before we reduce to computing the module structure on ${\rm Hom}_{gr R}(D_{t}^{+}, U)$. We of course have ${\rm dim}_k {\rm Hom}_{gr R}(D_t^{+}, k) = 1$ via the canonical quotient $D_t^{+} \twoheadrightarrow k$ and one easily verifies that the remaining ${\rm Hom}$ spaces are $0$ or $1$ dimensional with bases given by
\[
\xymatrix@C1pc{ R(-1) \ar@/^0.1pc/[r]^{x-y} \ar[d]^{\begin{bsmallmatrix}  1 \\ 0 \end{bsmallmatrix}} & R \ar[d]^{\begin{bsmallmatrix} 1 \end{bsmallmatrix}} \ar@{->>}[r] & D_0^+ \ar@{.>}[d] && R(-1) \ar@/^0.1pc/[r]^{x-y} \ar[d]^{\begin{bsmallmatrix}  0 \\ 0  \end{bsmallmatrix}} & R \ar[d]^{\begin{bsmallmatrix} 0 \end{bsmallmatrix}} \ar@{->>}[r] & D_0^+ \ar@{.>}[d] && R(-1) \ar@/^0.1pc/[r]^{x-y} \ar[d]^{\begin{bsmallmatrix}  0 \\ 0  \end{bsmallmatrix}} & R \ar[d]^{\begin{bsmallmatrix} 0 \end{bsmallmatrix}} \ar@{->>}[r] & D_0^+ \ar@{.>}[d] && R(-1) \ar@/^0.1pc/[r]^{x-y} \ar[d]^{\begin{bsmallmatrix} 1 \\ 0  \end{bsmallmatrix}} & R \ar[d]^{\begin{bsmallmatrix} 1 \end{bsmallmatrix}} \ar@{->>}[r] & D_0^+ \ar@{.>}[d]\\
R(-1)^{\oplus 2} \ar@/_0.1pc/[r]_{\begin{bsmallmatrix} x-y & y-z \end{bsmallmatrix}} & R \ar@{->>}[r] & L_1 && R(-1)^{\oplus 2} \ar@/_0.1pc/[r]_{\begin{bsmallmatrix} x+y & y-z \end{bsmallmatrix}} & R \ar@{->>}[r] & L_2 && R(-1)^{\oplus 2} \ar@/_0.1pc/[r]_{\begin{bsmallmatrix} x+y & y+z \end{bsmallmatrix}} & R \ar@{->>}[r] & L_3 && R(-1)^{\oplus 2} \ar@/_0.1pc/[r]_{\begin{bsmallmatrix} x-y & y+z \end{bsmallmatrix}} & R \ar@{->>}[r] & L_4 \\
}
\]
\[
\xymatrix@C1pc{ R(-1) \ar@/^0.1pc/[r]^{x+z} \ar[d]^{\begin{bsmallmatrix}  0 \\ 0 \end{bsmallmatrix}} & R \ar[d]^{\begin{bsmallmatrix} 0 \end{bsmallmatrix}} \ar@{->>}[r] & D_1^+ \ar@{.>}[d] && R(-1) \ar@/^0.1pc/[r]^{x+z} \ar[d]^{\begin{bsmallmatrix}  1 \\ -1  \end{bsmallmatrix}} & R \ar[d]^{\begin{bsmallmatrix} 1 \end{bsmallmatrix}} \ar@{->>}[r] & D_1^+ \ar@{.>}[d] && R(-1) \ar@/^0.1pc/[r]^{x+z} \ar[d]^{\begin{bsmallmatrix}  0 \\ 0  \end{bsmallmatrix}} & R \ar[d]^{\begin{bsmallmatrix} 0 \end{bsmallmatrix}} \ar@{->>}[r] & D_1^+ \ar@{.>}[d] && R(-1) \ar@/^0.1pc/[r]^{x+z} \ar[d]^{\begin{bsmallmatrix} 1 \\ 1 \end{bsmallmatrix}} & R \ar[d]^{\begin{bsmallmatrix} 1 \end{bsmallmatrix}} \ar@{->>}[r] & D_1^+ \ar@{.>}[d]\\
R(-1)^{\oplus 2} \ar@/_0.1pc/[r]_{\begin{bsmallmatrix} x-y & y-z \end{bsmallmatrix}} & R \ar@{->>}[r] & L_1 && R(-1)^{\oplus 2} \ar@/_0.1pc/[r]_{\begin{bsmallmatrix} x+y & y-z \end{bsmallmatrix}} & R \ar@{->>}[r] & L_2 && R(-1)^{\oplus 2} \ar@/_0.1pc/[r]_{\begin{bsmallmatrix} x+y & y+z \end{bsmallmatrix}} & R \ar@{->>}[r] & L_3 && R(-1)^{\oplus 2} \ar@/_0.1pc/[r]_{\begin{bsmallmatrix} x-y & y+z \end{bsmallmatrix}} & R \ar@{->>}[r] & L_4 \\
}
\]
\[
\xymatrix@C1pc{ R(-1) \ar@/^0.1pc/[r]^{y+z} \ar[d]^{\begin{bsmallmatrix}  0 \\ 0 \end{bsmallmatrix}} & R \ar[d]^{\begin{bsmallmatrix} 0 \end{bsmallmatrix}} \ar@{->>}[r] & D_\infty^+ \ar@{.>}[d] && R(-1) \ar@/^0.1pc/[r]^{y+z} \ar[d]^{\begin{bsmallmatrix}  0 \\ 0  \end{bsmallmatrix}} & R \ar[d]^{\begin{bsmallmatrix} 0 \end{bsmallmatrix}} \ar@{->>}[r] & D_\infty^+ \ar@{.>}[d] && R(-1) \ar@/^0.1pc/[r]^{y+z} \ar[d]^{\begin{bsmallmatrix}  0 \\ 1  \end{bsmallmatrix}} & R \ar[d]^{\begin{bsmallmatrix} 1 \end{bsmallmatrix}} \ar@{->>}[r] & D_\infty^+ \ar@{.>}[d] && R(-1) \ar@/^0.1pc/[r]^{y+z} \ar[d]^{\begin{bsmallmatrix} 0 \\ 1  \end{bsmallmatrix}} & R \ar[d]^{\begin{bsmallmatrix} 1 \end{bsmallmatrix}} \ar@{->>}[r] & D_\infty^+ \ar@{.>}[d]\\
R(-1)^{\oplus 2} \ar@/_0.1pc/[r]_{\begin{bsmallmatrix} x-y & y-z \end{bsmallmatrix}} & R \ar@{->>}[r] & L_1 && R(-1)^{\oplus 2} \ar@/_0.1pc/[r]_{\begin{bsmallmatrix} x+y & y-z \end{bsmallmatrix}} & R \ar@{->>}[r] & L_2 && R(-1)^{\oplus 2} \ar@/_0.1pc/[r]_{\begin{bsmallmatrix} x+y & y+z \end{bsmallmatrix}} & R \ar@{->>}[r] & L_3 && R(-1)^{\oplus 2} \ar@/_0.1pc/[r]_{\begin{bsmallmatrix} x-y & y+z \end{bsmallmatrix}} & R \ar@{->>}[r] & L_4 \\
}
\]

In this form one immediately sees that ${\rm E}(D_t^{+}) = {\rm Hom}_{gr R}(D_t^{+}, U) = S_t^{+}$, and ${\rm E}(D_t^{-}) \cong S_t^{-}$ follows from Prop \ref{S_t}, \ref{D_t} by using $\tau$.
\end{proof}

As corollary we obtain Prop \ref{N_t} through Theorem \ref{ARquiverMCM} from Prop \ref{R_t} through Theorem \ref{ARquiverQ}, since the representations $R_{t}$ for $t \in \mathbb{P}^1 \setminus \{0, 1, \infty \}$ and $S_{t}^{\pm}$ are the only simple regular representations of $Q$ and every module is constructed from them through unique, canonical extensions.
As we have already described said extensions in $\underline{{\rm MCM}}(gr R)$ in section $3$, our list of modules is exhaustive and we have classified every indecomposable graded MCM $R$-module.


\section{Relations with the preprojective algebra of type $\widetilde{D}_4$}

The preprojective algebra $\Pi(Q)$ of type $\widetilde{D}_4$ is given by the graded quiver path algebra $k\overline{Q}/I$ where $\overline{Q}$ is
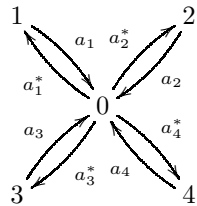
\[
\xymatrix@C1.8pc@R1.8pc{1 \ar@/^0.3pc/[dr]^{a_1} && 2 \ar@/^0.3pc/[dl]^{a_2}\\
            & 0 \ar@/^0.3pc/[ul]^{a_1^*} \ar@/^0.3pc/[ur]^{a_2^*} \ar@/^0.3pc/[dl]^{a_3^*} \ar@/^0.3pc/[dr]^{a_4^*} &\\
            3 \ar@/^0.3pc/[ur]^{a_3} && 4 \ar@/^0.3pc/[ul]^{a_4}
}
\]
with $|a_i| = 0$, $|a_i^*| = 1$ and $I = \big(\sum_{i = 1}^4 [a_i, a_i^*]\big)$, where $[-, -]$ is the commutator. One may describe $\Pi(Q)$ more invariantly as follows:
\[
\Pi(Q) \cong \bigoplus_{i \geq 0} {\rm Hom}_{{\rm D}^b(A)}(A, \tau^{-i}A).
\]
The latter is the orbit algebra of $A$ under the endofunctor $\tau^{-1}$, with multiplication given by the natural composition. This description easily transfers under the equivalence ${\rm E}: \underline{{\rm MCM}}(gr R)^{op} \xrightarrow{\cong} {\rm D}^b(A)$, noting that $\tau = (1)[-1]$ in $\underline{{\rm MCM}}(gr R)$ and that ${\rm E} \circ \tau = \tau^{-1} \circ {\rm E}$.
\begin{thm}\label{preproj} Let $U = \bigg( \bigoplus_{i = 1}^4 L_i \bigg) \oplus k$. There is an isomorphism of graded algebras $\Pi(Q) \cong {\rm Ext}_R^{{\rm diag}}(U, U)$. 
\end{thm}
\begin{proof}

We have an isomorphism of graded algebras
\begin{align*}
\bigoplus_{i \geq 0} \underline{{\rm Hom}}_{gr R}(T, \tau^{-i}T) & \cong \bigoplus_{i \geq 0} {\rm Hom}_{{\rm D}^b(A)}(\tau^iA, A) \\
& \cong \bigoplus_{i \geq 0} {\rm Hom}_{{\rm D}^b(A)}(A, \tau^{-i}A)\\
& = \Pi(Q).
\end{align*}
Now, consider $\bigoplus_{i \geq 0} {\rm Ext}^i_{gr R}(U, U(-i))$. By Orlov's Theorem the map ${\rm st}: {\rm Ext}^0_{gr R}(U, U) \xrightarrow{\cong} \underline{{\rm Hom}}_{gr R}(T, T)$ is bijective. For $i \geq 1$, since ${\rm dim}\ R = 1$ the map ${\rm st}: {\rm Ext}_{gr R}^{i}(U, U(-i)) \xrightarrow{\cong} \underline{{\rm Hom}}_{gr R}(T, T(-i)[i])$ is also bijective, and so
\begin{align*}
\bigoplus_{i \geq 0} \underline{{\rm Hom}}_{gr R}(T, \tau^{-i}T) & \cong \bigoplus_{i \geq 0} {\rm Ext}^{i, -i}_R(U, U) \\
& = {\rm Ext}^{{\rm diag}}_R(U, U)
\end{align*}
and the result follows.
\end{proof}
Note that the module $U$ is Koszul as alluded to in the abstract, in that it admits a linear resolution with generators in the same degrees. Theorem \ref{preproj} allows us to relate the Yoneda algebras of $L_i$ and $k$ to the preprojective algebra.
\begin{cor} Let $e_i \in \Pi(Q)$ be the primitive idempotent corresponding to the $i^{th}$ vertex. Then we have isomorphisms of graded algebras
\[
    e_i\Pi(Q)e_i \cong \begin{cases} {\rm Ext}^*_R(k, k) & i = 0 \\
     {\rm Ext}^{{\rm diag}}_R(L_i, L_i) & i = 1,2,3,4.\end{cases}
\]
\end{cor}
We have used that $R$ is Koszul, and so ${\rm Ext}_R^{{\rm diag}}(k, k) = {\rm Ext}_R^*(k, k)$. This has interesting consequences:

\begin{cor} The graded algebra $e_0 \Pi(Q) e_0$ is Koszul and ${\rm Ext}_{e_0 \Pi(Q) e_0}^*(k, k) \cong R$ as graded algebras. Moreover $e_0 \Pi(Q) e_0 \cong \mathcal{U}(\mathcal{L})$ is the universal enveloping algebra of the homotopy Lie algebra $\mathcal{L} = \pi(R)$.
\end{cor} 
These are well-known consequences of Koszul duality for associative (resp. commutative and lie) algebras. See \cite{MR1648664}, \cite{BGS} for the definitions and results in this context.

\begin{cor} Assume that $k$ has characteristic zero. Then ${\rm Ext}_R^{{\rm diag}}(L_i, L_i) \cong k[u,v]^{\Gamma}$ is a Kleinian singularity of type $D_4$, with $\Gamma \le SL_2(k)$ the corresponding finite subgroup.
\end{cor}
Indeed, Crawley-Boevey and Holland have shown in \cite{MR1620538} that if $Q$ is an extended ADE graph with extended vertex $e$, then $e \Pi(Q) e \cong k[u,v]^{\Gamma}$ for $\Gamma \le SL_2(k)$ a finite subgroup of same ADE type under the McKay correspondence.

\bibliographystyle{plain}
\bibliography{conicsbib}

\end{document}